\documentclass[11pt,a4paper]{amsart}
\usepackage{amssymb}
\usepackage{amsmath}
\usepackage{amsfonts}

\newtheorem{theorem}{Theorem}
\theoremstyle{plain}

\newtheorem{corollary}{Corollary}

\newtheorem{definition}{Definition}
\newtheorem{example}{Example}

\newtheorem{lemma}{Lemma}
\newtheorem{notation}{Notation}

\newtheorem{proposition}{Proposition}
\newtheorem{remark}{Remark}

\numberwithin{equation}{section}
\input{tcilatex}

\begin{document}
\title{Global attractors for semigroup actions on uniformizable spaces}
\author{Josiney A. Souza and Richard W. M. Alves}
\address{Departamento de Matem\'{a}tica, Universidade Estadual de Maring\'{a}%
, Brasil.}
\email{jasouza3@uem.br}

\begin{abstract}
In this paper the notion of global attractor is extended from the setting of
semigroup actions on metric spaces to the setting of semigroup actions on
uniformizable spaces. General conditions for the existence of global
attractor are discussed and its relationship with the global uniform
attractor is presented.
\end{abstract}

\maketitle

\section{Introduction}

In this paper we extend the notion of global attractor to the setting of
semigroup actions on uniformizable spaces. This setting covers a wide scope
of mathematical analysis, since almost every topological space studied in
this branch of mathematics is at least uniformizable (completely regular):
metric and pseudometric spaces, locally compact regular spaces, manifolds,
topological groups, Lie groups, paracompact regular spaces, normal regular
spaces, etc. Besides, semigroup actions include semigroups of operators,
multi-time dynamical systems, polisystems, control systems, Lie group
actions, Ellis actions, etc.. Then studying global attractors in this
general setting contributes to the branch of mathematical analysis and
topological dynamics.

The concept of global attractor for general semigroup actions was defined in 
\cite{Stephanie}. Let $S$ be a semigroup acting on a metric space $M$ and $%
\mathcal{F}$ a filter basis on the subsets of $S$. The global $\mathcal{F}$%
-attractor for $\left( S,M\right) $ is an invariant compact subset $\mathcal{%
A}$ of $M$ that attracts every bounded subset $Y$ of $M$, which means that 
\begin{equation*}
\lim_{t_{\lambda }\rightarrow _{\mathcal{F}}\infty }\mathrm{dist}\left(
t_{\lambda }Y;\mathcal{A}\right) =0
\end{equation*}%
for every $\mathcal{F}$-divergent net $\left( t_{\lambda }\right) _{\lambda
\in \Lambda }$ in $S$. If the global $\mathcal{F}$-attractor exists then its
region of uniform $\mathcal{F}$-attraction $\mathrm{A}_{u}\left( \mathcal{A},%
\mathcal{F}\right) $ coincides with the whole space $M$, that is, 
\begin{equation*}
M=\mathrm{A}_{u}\left( \mathcal{A},\mathcal{F}\right) =\left\{ x\in
M:J\left( x,\mathcal{F}\right) \neq \emptyset \text{ and }J\left( x,\mathcal{%
F}\right) \subset \mathcal{A}\right\}
\end{equation*}%
where $J\left( x,\mathcal{F}\right) $ is the forward $\mathcal{F}$%
-prolongational limit set of $x$. Thus the global $\mathcal{F}$-attractor is
the global uniform $\mathcal{F}$-attractor for $\left( S,M\right) $. The
converse holds if the family $\mathcal{F}$ satisfies certain translation
hypothesis and the action is eventually compact and $\mathcal{F}$%
-asymptotically compact (\cite[Theorem 3]{Stephanie}).

Global uniform attractors for bitransformation semigroups were recently
studied in \cite{So4}. In this work, the phase space is a general
topological space, since the concept of global uniform attractor does not
depend on a metric. We have discovered that the notion of global attractor
defined in \cite{Stephanie} is also independent of metrization on the phase
space and it can be defined in completely regular spaces.

For defining attraction for a semigroup action on completely regular space,
we consider an admissible structure instead of a uniform structure on the
phase space (\cite{Richard}). In addition, we reproduce the notions of
Hausdorff semidistance and the measure of noncompactness. Then we extend all
results of \cite{Stephanie}, discussing necessary and sufficient conditions
for the existence of the global attractor: eventual compactness, eventual
boundedness, bounded dissipativeness, and asymptotic compactness. We also
introduce the notion of limit compact semigroup action and use the
Cantor--Kuratowski theorem, proved in \cite{Richard2}, to provide a
relationship between asymptotic compactness and limit compactness.

To illustrate problems in the general setting of this paper, we provide
examples of semigroup actions on function spaces, which are fundamental
components of mathematical analysis. For while, let $E^{E}$ be the function
space of a normed vector space $E$ with the pointwise convergence topology.
By considering the multiplicative semigroup of positive integers $\mathbb{N}$%
, let $\mu :\mathbb{N}\times E^{E}\rightarrow E^{E}$ be the action defined
as $\mu \left( n,f\right) =f^{n}$. Then $\mu $ defines a semigroup action
that is not a classical semiflow and $E^{E}$ is not metrizable. However, $%
E^{E}$ is uniformizable with the uniformity of pointwise convergence. The
subspace $X\subset E^{E}$ of all contraction operators with common Lipschitz
constant $L<1$ is $\mu $-invariant. The trivial operator $T_{0}\equiv 0$ is
the global attractor for the restricted action $\mu :\mathbb{N}\times
X\rightarrow X$ (Example \ref{Ex4}).

\section{Preliminaries on admissible spaces}

This section contains the basic definitions and properties of admissible
spaces. We refer to \cite{Richard}, \cite{patrao2}, and \cite{So} for the
previous development of admissible spaces.

Let $X$ be a completely regular space and $\mathcal{U},\mathcal{V}$
coverings of $X$. We write $\mathcal{V}\leqslant \mathcal{U}$ if $\mathcal{V}
$ is a refinement of $\mathcal{U}$. One says $\mathcal{V}$ double-refines $%
\mathcal{U}$, or $\mathcal{V}$ is a double-refinement of $\mathcal{U}$,
written $\mathcal{V}\leqslant \frac{1}{2}\mathcal{U}$ or $2\mathcal{V}%
\leqslant \mathcal{U}$, if for every $V,V^{\prime }\in \mathcal{V}$, with $%
V\cap V^{\prime }\neq \emptyset $, there is $U\in \mathcal{U}$ such that $%
V\cup V^{\prime }\subset U$. We write $\mathcal{V}\leqslant \frac{1}{2^{2}}%
\mathcal{U}$ if there is a covering $\mathcal{W}$ of $X$ such that $\mathcal{%
V}\leqslant \frac{1}{2}\mathcal{W}$ and $\mathcal{W}\leqslant \frac{1}{2}%
\mathcal{U}$. Inductively, we write $\mathcal{V}\leqslant \frac{1}{2^{n}}%
\mathcal{U}$ if there is $\mathcal{W}$ with $\mathcal{V}\leqslant \frac{1}{2}%
\mathcal{W}$ and $\mathcal{W}\leqslant \frac{1}{2^{n-1}}\mathcal{U}$.

For a covering $\mathcal{U}$ of $X$ and a subset $Y\subset X$, the \emph{star%
} of $Y$ with respect to $\mathcal{U}$ is the set 
\begin{equation*}
\mathrm{St}\left[ Y,\mathcal{U}\right] =\bigcup \left\{ U\in \mathcal{U}%
:Y\cap U\neq \emptyset \right\} \text{.}
\end{equation*}%
If $Y=\left\{ x\right\} $, we usually write $\mathrm{St}\left[ x,\mathcal{U}%
\right] $ rather than $\mathrm{St}\left[ \left\{ x\right\} ,\mathcal{U}%
\right] $. Then one has $\mathrm{St}\left[ Y,\mathcal{U}\right]
=\bigcup\limits_{x\in Y}\mathrm{St}\left[ x,\mathcal{U}\right] $ for every
subset $Y\subset X$.

\begin{definition}
\label{Admiss} A family $\mathcal{O}$ of open coverings of $X$ is said to be 
\textbf{admissible} if it satisfies the following properties:

\begin{enumerate}
\item For each $\mathcal{U}\in \mathcal{O}$, there is $\mathcal{V}\in 
\mathcal{O}$ such that $\mathcal{V}\leqslant \frac{1}{2}\mathcal{U}$;

\item If $Y\subset X$ is an open set and $K\subset Y$ is a compact subset of 
$X$ then there is an open covering $\mathcal{U}\in \mathcal{O}$ such that $%
\mathrm{St}\left[ K,\mathcal{U}\right] \subset Y$;

\item For any $\mathcal{U},\mathcal{V}\in \mathcal{O}$, there is $\mathcal{W}%
\in \mathcal{O}$ such that $\mathcal{W}\leqslant \mathcal{U}$ and $\mathcal{W%
}\leqslant \mathcal{V}$.
\end{enumerate}
\end{definition}

The properties 1 and 2 of Definition \ref{Admiss} guarantee that the stars $%
\mathrm{St}\left[ x,\mathcal{U}\right] $, for $x\in X$ and $\mathcal{U}\in 
\mathcal{O}$, form a basis for the topology of $X$, while the property 3
allows to direct $\mathcal{O}$ by refinements. Definition \ref{Admiss} can
be simplified by requiring that the family $\mathcal{O}$ of open coverings
of $X$ satisfies the following:

\begin{enumerate}
\item[$\left( 1\right) ^{\prime }$] For any $\mathcal{U},\mathcal{V}\in 
\mathcal{O}$, there is $\mathcal{W}\in \mathcal{O}$ such that $\mathcal{W}%
\leqslant \frac{1}{2}\mathcal{U}$ and $\mathcal{W}\leqslant \frac{1}{2}%
\mathcal{V}$.

\item[$\left( 2\right) ^{\prime }$] The stars $\mathrm{St}\left[ x,\mathcal{U%
}\right] $, for $x\in X$ and $\mathcal{U}\in \mathcal{O}$, form a basis for
the topology of $X$.
\end{enumerate}

\begin{remark}
Let $\mathcal{O}$ be an admissible family of open coverings of $X$. For
every set $Y\subset X$, one has $\mathrm{cls}\left( Y\right) =\underset{%
\mathcal{U}\in \mathcal{O}}{\bigcap }\mathrm{St}\left[ Y,\mathcal{U}\right] $%
.
\end{remark}

\begin{definition}
\label{Full}An admissible family $\mathcal{O}$ of open coverings of $X$ is
said to be \textbf{replete} if it satisfies the following additional
conditions:

\begin{enumerate}
\item $X=\bigcup\limits_{\mathcal{U}\in \mathcal{O}}\mathrm{St}\left[ x,%
\mathcal{U}\right] $ for every $x\in X$.

\item For any $\mathcal{U},\mathcal{V}\in \mathcal{O}$, there is $\mathcal{W}%
\in \mathcal{O}$ such that $\mathcal{U}\leqslant \frac{1}{2}\mathcal{W}$ and 
$\mathcal{V}\leqslant \frac{1}{2}\mathcal{W}$.
\end{enumerate}
\end{definition}

\begin{example}
\label{Ex1}

\begin{enumerate}
\item Every covering uniformity $\mathcal{O}$ of $X$ is a replete admissible
family.

\item If $X$ is a paracompact Hausdorff space then the family $\mathcal{O}%
\left( X\right) $ of all open coverings of $X$ is a replete admissible
family.

\item If $X$ is a compact Hausdorff space then the family $\mathcal{O}_{f}$
of all finite open coverings of $X$ is a replete admissible family.

\item If $\left( X,\mathrm{d}\right) $ is a metric space then the family $%
\mathcal{O}_{\mathrm{d}}$ of the coverings $\mathcal{U}_{\varepsilon
}=\left\{ \mathrm{B}_{\mathrm{d}}\left( x,\varepsilon \right) :x\in
X\right\} $ by $\varepsilon $-balls, for $\varepsilon >0$, is a replete
admissible family.
\end{enumerate}
\end{example}

If $\mathcal{O}$ is an admissible family of $X$ that is upwards hereditary
then $\mathcal{O}$ is clearly replete (Example \ref{Ex1}, items 1 and 2).
However, a replete admissible family need not be upwards hereditary (Example %
\ref{Ex1}, items 3 and 4). In general, for a given admissible family $%
\mathcal{O}$ of $X$, we may consider the replete admissible family $%
\widetilde{\mathcal{O}}$, generated by $\mathcal{O}$, given by 
\begin{equation*}
\widetilde{\mathcal{O}}=\left\{ \mathcal{V}:\mathcal{V}\text{ is open
covering of }X\text{ and }\mathcal{U}\leqslant \mathcal{V}\text{ for some }%
\mathcal{U}\in \mathcal{O}\right\} .
\end{equation*}

From now on, $X$ is a fixed completely regular space endowed with a replete
admissible family of open coverings $\mathcal{O}$. Let $\mathcal{P}\left( 
\mathcal{O}\right) $ denote the power set of $\mathcal{O}$ and consider the
partial ordering relation on $\mathcal{P}\left( \mathcal{O}\right) $ given
by inverse inclusion: for $\mathcal{E}_{1},\mathcal{E}_{2}\in \mathcal{P}%
\left( \mathcal{O}\right) $%
\begin{equation*}
\mathcal{E}_{1}\prec \mathcal{E}_{2}\text{ if and only if }\mathcal{E}%
_{1}\supset \mathcal{E}_{2}.
\end{equation*}%
Concerning this relation, $\mathcal{O}$ is the smallest element in $\mathcal{%
P}\left( \mathcal{O}\right) $, or in other words, $\mathcal{O}$ is the lower
bound for $\mathcal{P}\left( \mathcal{O}\right) $. On the other hand, the
empty set $\emptyset $ is the upper bound for $\mathcal{P}\left( \mathcal{O}%
\right) $. Intuitively, $\mathcal{O}$ is the \textquotedblleft
zero\textquotedblright\ and $\emptyset $ is the \textquotedblleft
infinity\textquotedblright . For each $\mathcal{E}\in \mathcal{P}\left( 
\mathcal{O}\right) $ and $n\in \mathbb{N}^{\ast }$ we define the set $n%
\mathcal{E}$ in $\mathcal{P}\left( \mathcal{O}\right) $ by 
\begin{equation*}
n\mathcal{E}=\left\{ \mathcal{U}\in \mathcal{O}:\text{there is }\mathcal{V}%
\in \mathcal{E}\text{ such that }\mathcal{V}\leqslant \tfrac{1}{2^{n}}%
\mathcal{U}\right\} .
\end{equation*}%
This operation is order-preserving, that is, if $\mathcal{E}\prec \mathcal{D}
$ then $n\mathcal{E}\prec n\mathcal{D}$. In fact, if $\mathcal{U}\in n%
\mathcal{D}$ then there is $\mathcal{V}\in \mathcal{D}$ such that $\mathcal{V%
}\leqslant \tfrac{1}{2^{n}}\mathcal{U}$. As $\mathcal{D}\subset \mathcal{E}$%
, it follows that $\mathcal{U}\in n\mathcal{E}$, and therefore $n\mathcal{E}%
\prec n\mathcal{D}$. Note also that $n\mathcal{O}=\mathcal{O}$, for every $%
n\in \mathbb{N}^{\ast }$, since for each $\mathcal{U}\in \mathcal{O}$ there
is $\mathcal{V}\in \mathcal{O}$ such that $\mathcal{V}\leqslant \tfrac{1}{%
2^{n}}\mathcal{U}$, that is, $\mathcal{U}\in n\mathcal{O}$.

We often consider the following notion of convergence in $\mathcal{P}\left( 
\mathcal{O}\right) $.

\begin{definition}
\label{Convergence}We say that a net $\left( \mathcal{E}_{\lambda }\right) $
in $\mathcal{P}\left( \mathcal{O}\right) $ \textbf{converges} to $\mathcal{O}
$, written $\mathcal{E}_{\lambda }\rightarrow \mathcal{O}$,\ if for every $%
\mathcal{U}\in \mathcal{O}$ there is a $\lambda _{0}$ such that $\mathcal{U}%
\in \mathcal{E}_{\lambda }$ whenever $\lambda \geq \lambda _{0}$.
\end{definition}

It is easily seen that $\mathcal{D}_{\lambda }\prec \mathcal{E}_{\lambda }$
and $\mathcal{E}_{\lambda }\rightarrow \mathcal{O}$ implies $\mathcal{D}%
_{\lambda }\rightarrow \mathcal{O}$. Moreover, $\mathcal{E}_{\lambda
}\rightarrow \mathcal{O}$ implies $n\mathcal{E}_{\lambda }\rightarrow 
\mathcal{O}$ for every $n\in \mathbb{N}^{\ast }$ (see \cite[Proposition 1]%
{Richard2}).

We also need the auxiliary function $\rho :X\times X\rightarrow \mathcal{P}%
\left( \mathcal{O}\right) $ given by 
\begin{equation*}
\rho \left( x,y\right) =\left\{ \mathcal{U}\in \mathcal{O}:y\in \mathrm{St}%
\left[ x,\mathcal{U}\right] \right\} .
\end{equation*}%
Note that the value $\rho \left( x,y\right) $ is upwards hereditary, that
is, if $\mathcal{U}\leqslant \mathcal{V}$ with $\mathcal{U}\in \rho \left(
x,y\right) $ then $\mathcal{V}\in \rho \left( x,y\right) $. The following
properties of the function $\rho $ are proved in \cite[Propositions 2 and 3]%
{Richard2}.

\begin{proposition}
\label{P1}

\begin{enumerate}
\item $\rho \left( x,y\right) =\rho \left( y,x\right) $ for all $x,y\in X$.

\item $\mathcal{O}\prec \rho \left( x,y\right) $, for all $x,y\in X$, and $%
\mathcal{O}=\rho \left( x,x\right) $.

\item If $X$ is Hausdorff, $\mathcal{O}=\rho \left( x,y\right) $ if and only
if $x=y$.

\item $\rho \left( x,y\right) \prec n\left( \rho \left( x,x_{1}\right) \cap
\rho \left( x_{1},x_{2}\right) \cap \ldots \cap \rho \left( x_{n},y\right)
\right) $ for all $x,y,x_{1},...,x_{n}\in X$.

\item A net $\left( x_{\lambda }\right) $ in $X$ converges to $x$ if and
only if $\rho \left( x_{\lambda },x\right) \rightarrow \mathcal{O}$.
\end{enumerate}
\end{proposition}

For a given point $x\in X$ and a set $A\subset X$, we define the set $\rho
\left( x,A\right) \in \mathcal{P}\left( \mathcal{O}\right) $ by%
\begin{equation*}
\rho \left( x,A\right) =\bigcup\limits_{y\in A}\rho \left( x,y\right) .
\end{equation*}%
For two nonempty subsets $A,B\subset X$, we define the collection $\rho
_{A}\left( B\right) \in \mathcal{P}\left( \mathcal{O}\right) $ by 
\begin{equation*}
\rho _{A}\left( B\right) =\bigcap\limits_{b\in B}\rho \left( b,A\right)
=\bigcap\limits_{b\in B}\bigcup\limits_{a\in A}\rho \left( b,a\right) .
\end{equation*}

It is easily seen that $\rho \left( x,A\right) \prec \rho \left( x,a\right) $%
, for all $a\in A$, and $\rho \left( x,B\right) \prec \rho \left( x,A\right) 
$ whenever $B\supset A$.

\begin{proposition}
\label{R4}For a given point $x\in X$ and subsets $A,B\subset X$, the
following properties hold:

\begin{enumerate}
\item $\rho \left( x,A\right) =\mathcal{O}$ if and only if $x\in \mathrm{cls}%
\left( A\right) $.

\item $\rho \left( x,\mathrm{cls}\left( A\right) \right) =\rho \left(
x,A\right) .$

\item $\rho _{A}\left( B\right) =\rho _{\mathrm{cls}\left( A\right) }\left(
B\right) .$

\item $B\subset \mathrm{cls}\left( A\right) $ if and only if $\rho
_{A}\left( B\right) =\mathcal{O}$.
\end{enumerate}
\end{proposition}

\begin{proof}
For item $\left( 1\right) $, note that $\rho \left( x,A\right) =\mathcal{O}$
if and only if $A\cap \mathrm{St}\left[ x,\mathcal{U}\right] \neq \emptyset $
for all $\mathcal{U}\in \mathcal{O}$. Since the collection $\left\{ \mathrm{%
St}\left[ x,\mathcal{U}\right] :\mathcal{U}\in \mathcal{O}\right\} $ is a
neighborhood base at $x$, it follows that $\rho \left( x,A\right) =\mathcal{O%
}$ if and only if $x\in \mathrm{cls}\left( A\right) $. For item $\left(
2\right) $, note that if $\mathcal{U}\in \rho \left( x,\mathrm{cls}\left(
A\right) \right) $ then $\mathcal{U}\in \rho \left( x,y\right) $ for some $%
y\in \mathrm{cls}\left( A\right) $. Hence $y\in \mathrm{St}\left[ x,\mathcal{%
U}\right] \cap \mathrm{cls}\left( A\right) $, and therefore $\mathrm{St}%
\left[ x,\mathcal{U}\right] \cap A\neq \emptyset $ because $\mathrm{St}\left[
x,\mathcal{U}\right] $ is open. It follows that $\mathcal{U}\in \rho \left(
x,a\right) $ for some $a\in A$. Thus $\mathcal{U}\in \rho \left( x,A\right) $
and we have the inclusion $\rho \left( x,\mathrm{cls}\left( A\right) \right)
\subset \rho \left( x,A\right) $. The inclusion $\rho \left( x,A\right)
\subset \rho \left( x,\mathrm{cls}\left( A\right) \right) $ is clear. Items $%
\left( 3\right) $ and $\left( 4\right) $ follow by item $\left( 2\right) $.
\end{proof}

\begin{proposition}
\label{R5} Let $A\subset X$ be a nonempty subset and take a convergent net $%
x_{\lambda }\rightarrow x$ in $X$. Then $x\in \mathrm{cls}\left( A\right) $
if and only if $\rho _{A}\left( x_{\lambda }\right) \rightarrow \mathcal{O}$.
\end{proposition}

\begin{proof}
Suppose that $x\in \mathrm{cls}\left( A\right) $ and take any $\mathcal{U}%
\in \mathcal{O}$. By Proposition \ref{P1}, item 5, there is $\lambda _{0}$
such that $\lambda \geq \lambda _{0}$ implies $\mathcal{U}\in \rho \left(
x_{\lambda },x\right) $. Hence 
\begin{equation*}
\mathcal{U}\in \rho \left( x_{\lambda },x\right) \subset \rho _{\mathrm{cls}%
\left( A\right) }\left( x_{\lambda }\right) =\rho _{A}\left( x_{\lambda
}\right)
\end{equation*}%
and therefore $\rho _{A}\left( x_{\lambda }\right) \rightarrow \mathcal{O}$.
On the other hand, suppose that $\rho _{A}\left( x_{\lambda }\right)
\rightarrow \mathcal{O}$ and take any $\mathcal{U}\in \mathcal{O}$. For $%
\mathcal{V}\in \mathcal{O}$ with $\mathcal{V}\leqslant \frac{1}{2}\mathcal{U}
$, there is $\lambda _{0}$ such that $\lambda \geq \lambda _{0}$ implies $%
\mathcal{V}\in \rho _{A}\left( x_{\lambda }\right) \cap \rho \left(
x_{\lambda },x\right) $. Then, for each $\lambda \geq \lambda _{0}$, there
is some $a_{\lambda }\in A$ such that $\mathcal{V}\in \rho \left( x_{\lambda
},a_{\lambda }\right) $. Hence 
\begin{equation*}
\rho _{A}\left( x\right) \prec \rho \left( x,a_{\lambda }\right) \prec
1\left( \rho \left( x,x_{\lambda }\right) \cap \rho \left( x_{\lambda
},a_{\lambda }\right) \right) \prec 1\left\{ \mathcal{V}\right\} \prec
\left\{ \mathcal{U}\right\} .
\end{equation*}%
This means that $\rho _{A}\left( x\right) =\mathcal{O}$, and therefore $x\in 
\mathrm{cl}\left( A\right) $.
\end{proof}

We now define the notions of bounded set and totally bounded set.

\begin{definition}
\label{Totallybounded}A nonempty subset $Y\subset X$ is called\ \textbf{%
bounded} with respect to $\mathcal{O}$ if there is some $\mathcal{U}\in 
\mathcal{O}$ such that $\mathcal{U}\in \rho \left( x,y\right) $ for all $%
x,y\in Y$; $Y$ is said to be \textbf{totally bounded} if for each $\mathcal{U%
}\in \mathcal{O}$ there is a finite sequence $x_{1},...,x_{n}$ such that $%
Y\subset \bigcup\limits_{i=1}^{n}\mathrm{St}\left[ x_{i},\mathcal{U}\right] $%
.
\end{definition}

\begin{remark}
\label{R1}A totally bounded set is bounded. Indeed, assume that $Y\subset X$
is totally bounded. For a given $\mathcal{U}\in \mathcal{O}$, there is a
finite cover $Y\subset \bigcup\limits_{i=1}^{n}\mathrm{St}\left[ x_{i},%
\mathcal{U}\right] $. Since $\mathcal{O}$ is replete, we can take $\mathcal{V%
}\in \mathcal{O}$ such that $\mathcal{U}\leqslant \mathcal{V}$ and $\mathcal{%
V}\in \rho \left( x_{i},x_{j}\right) $ for every pair $i,j\in \left\{
1,...,n\right\} $. For $x,y\in Y$, we have $x\in \mathrm{St}\left[ x_{i},%
\mathcal{U}\right] $ and $y\in \mathrm{St}\left[ x_{j},\mathcal{U}\right] $
for some $i,j$, hence 
\begin{equation*}
\rho \left( x,y\right) \prec 2\left( \rho \left( x,x_{i}\right) \cap \rho
\left( x_{i},x_{j}\right) \cap \rho \left( x_{j},y\right) \right) \prec
2\left\{ \mathcal{V}\right\}
\end{equation*}%
with $2\left\{ \mathcal{V}\right\} \neq \emptyset $.
\end{remark}

\begin{remark}
\label{R2}If $A\subset X$ is bounded then $\mathrm{St}\left[ A,\mathcal{U}%
\right] $ is also bounded, for any $\mathcal{U}\in \mathcal{O}$. In fact,
take $\mathcal{V}\in \mathcal{O}$ such that $\mathcal{V}\in \rho \left(
x,y\right) $ for all $x,y\in A$. If $u,v\in \mathrm{St}\left[ A,\mathcal{U}%
\right] $ then there is $x,y\in A$ such that $u\in \mathrm{St}\left[ x,%
\mathcal{U}\right] $ and $v\in \mathrm{St}\left[ y,\mathcal{U}\right] $.
Take $\mathcal{W}\in \mathcal{O}$ such that both coverings $\mathcal{U}$ and 
$\mathcal{V}$ refine $\mathcal{W}$. We have $\rho \left( u,v\right) \prec
2\left( \rho \left( u,x\right) \cap \rho \left( x,y\right) \cap \rho \left(
y,v\right) \right) \prec 2\left\{ \mathcal{W}\right\} $, and therefore $%
\mathrm{St}\left[ A,\mathcal{U}\right] $ is bounded.
\end{remark}

The difference between totally bounded set and compact set is the
completeness, as the following.

\begin{definition}
A net $\left( x_{\lambda }\right) _{\lambda \in \Lambda }$ in $X$ is $%
\mathcal{O}$\textbf{-Cauchy} (or just \textbf{Cauchy}) if for each $\mathcal{%
U}\in \mathcal{O}$ there is some $\lambda _{0}\in \Lambda $ such that $%
x_{\lambda _{1}}$ and $x_{\lambda _{2}}$ lie together in some element of $%
\mathcal{U}$, that is, $\mathcal{U}\in \rho \left( x_{\lambda
_{1}},x_{\lambda _{2}}\right) $, whenever $\lambda _{1},\lambda
_{2}\geqslant \lambda _{0}$. If every Cauchy net in $X$ converges then $X$
is called a \textbf{complete admissible space}.
\end{definition}

The space $X$ is compact if and only if it is complete and totally bounded (%
\cite[Theorem 2]{Richard2}).

We now define the notion of measure of noncompactness.

\begin{definition}
Let $Y\subset X$ be a nonempty set. The \textbf{star measure of
noncompactness} of $Y$ is the set $\alpha \left( Y\right) \in \mathcal{P}%
\left( \mathcal{O}\right) $ defined as 
\begin{equation*}
\alpha \left( Y\right) =\bigcup \left\{ \mathcal{U}\in \mathcal{O}:Y\text{
admits a finite cover }Y\subset \bigcup\limits_{i=1}^{n}\mathrm{St}\left[
x_{i},\mathcal{U}\right] \right\} .
\end{equation*}
\end{definition}

If $Y\subset X$ is a bounded set then $\alpha \left( Y\right) $ is nonempty.
The following properties of measure of noncompactness are proved in \cite[%
Proposition 10]{Richard2}.

\begin{proposition}
\label{P9}

\begin{enumerate}
\item $\alpha \left( Y\right) \prec \gamma \left( Y\right) \prec 1\alpha
\left( Y\right) .$

\item $\alpha \left( Y\right) \prec \alpha \left( Z\right) $ if $Y\subset Z$.

\item $\alpha \left( Y\cup Z\right) =\alpha \left( Y\right) \cap \alpha
\left( Y\right) .$

\item $\alpha \left( Y\right) \prec \alpha \left( \mathrm{cls}\left(
Y\right) \right) \prec 1\alpha \left( Y\right) .$
\end{enumerate}
\end{proposition}

If $\mathrm{cls}\left( Y\right) $ is compact then $\alpha \left( Y\right) =%
\mathcal{O}$. The converse holds if $X$ is a complete admissible space (see 
\cite[Proposition 11]{Richard2}). The following theorem is proved in \cite%
{Richard2}. A decreasing net $\left( F_{\lambda }\right) $ of nonempty
closed sets of $X$ means that $F_{\lambda _{1}}\subset F_{\lambda _{2}}$
whenever $\lambda _{1}\geq \lambda _{2}$.

\begin{theorem}[Cantor--Kuratowski theorem]
\label{CantorKuratowski}The admissible space $X$ is complete if and only if
every decreasing net $\left( F_{\lambda }\right) $ of nonempty bounded
closed sets of $X$, with $\alpha \left( F_{\lambda }\right) \rightarrow 
\mathcal{O}$, has nonempty compact intersection.
\end{theorem}

\section{Limit behavior of semigroup actions}

This section contains basic definitions and properties of limit behavior of
semigroup actions on topological spaces. We refer to \cite{Stephanie} for
previous studies on dynamical concepts of semigroup actions. Throughout,
there is a fixed completely regular space $X$ endowed with and admissible of
open coverings $\mathcal{O}$.

Let $S$ be a topological semigroup. An \emph{action} of $S$ on $X$ is a
mapping

\begin{equation*}
\mu :%
\begin{array}[t]{ccc}
S\times X & \rightarrow & X \\ 
(s,x) & \mapsto & \mu (s,x)=sx%
\end{array}%
\end{equation*}%
satisfying $s\left( tx\right) =\left( st\right) x$ for all $x\in X$ and $%
s,t\in S$. We denote by $\mu _{s}:X\rightarrow X$ the map $\mu _{s}\left(
\cdot \right) =\mu \left( s,\cdot \right) $. The action is said to be \emph{%
open} if every $\mu _{s}$ is an open map; the action is called \emph{%
surjective} if every $\mu _{s}$ is surjective; and the action is said to be 
\emph{eventually compact} if there is some $t\in S$ such that $\mu
_{t}:X\rightarrow X$ is a compact map, that is, for every bounded set $%
Y\subset X$ the image $\mu _{t}\left( Y\right) $ is relatively compact. In
this paper we assume that $\mu _{s}$ is continuous for all $s\in S$. We
often indicate the semigroup action as $\left( S,X,\mu \right) $, or simply $%
\left( S,X\right) $.

For subsets $Y\subset X$ and $A\subset S$ we define the set $%
AY=\bigcup\limits_{s\in A}\mu _{s}\left( Y\right) $. We say that $Y$ is 
\emph{invariant} if $\mu _{s}\left( Y\right) =sY=Y$ for every $s\in S$.

For limit behavior of $\left( S,X\right) $, we assume a fixed filter basis $%
\mathcal{F}$ on the subsets of $S$ (that is, $\emptyset \notin \mathcal{F}$
and given $A,B\in \mathcal{F}$ there is $C\in \mathcal{F}$ with $C\subset
A\cap B$). The following notion of $\mathcal{F}$-divergent net was defined
in \cite{Stephanie}.

\begin{notation}
\label{Note}A net $\left( t_{\lambda }\right) _{\lambda \in \Lambda }$ in $S$
\textbf{diverges} on the direction of $\mathcal{F}$ ($\mathcal{F}$\textbf{%
-diverges}) if for each $A\in \mathcal{F}$ there is $\lambda _{0}\in \Lambda 
$ such that $t_{\lambda }\in A$ for all $\lambda \geq \lambda _{0}$. The
notation $t_{\lambda }\longrightarrow _{\mathcal{F}}\infty $ means that $%
\left( t_{\lambda }\right) $ $\mathcal{F}$-diverges.
\end{notation}

See \cite{Stephanie} for examples and general explanations on $\mathcal{F}$%
-divergent nets.

\begin{definition}
The $\omega $\textbf{-limit set} of $Y\subset X$ on the direction of $%
\mathcal{F}$ is defined as 
\begin{equation*}
\omega \left( Y,\mathcal{F}\right) =\bigcap_{A\in \mathcal{F}}\mathrm{cls}%
\left( AY\right) =\left\{ 
\begin{array}{c}
x\in X:\text{ there are nets }\left( t_{\lambda }\right) _{\lambda \in
\Lambda }\text{ in }S\text{ and }\left( x_{\lambda }\right) _{\lambda \in
\Lambda }\text{ in }Y \\ 
\text{such that }t_{\lambda }\longrightarrow _{\mathcal{F}}\infty \text{ and 
}t_{\lambda }x_{\lambda }\longrightarrow x%
\end{array}%
\right\}
\end{equation*}
\end{definition}

\begin{remark}
If $Y\subset X$ is invariant then $\omega \left( Y,\mathcal{F}\right) =%
\mathrm{cls}\left( Y\right) $. If $Y$ is also closed then $\omega \left( Y,%
\mathcal{F}\right) =Y$.
\end{remark}

\begin{definition}
\label{FirstProlong} For a given point $x\in X$, the \textbf{forward }$%
\mathcal{F}$-\textbf{prolongational limit set} of $x\in X$ is defined as%
\begin{equation*}
J\left( x,\mathcal{F}\right) =\left\{ 
\begin{array}{c}
y\in X:\text{there are nets }\left( t_{\lambda }\right) \text{ in }S\text{
and }\left( x_{\lambda }\right) \text{ in }X\text{ such that } \\ 
t_{\lambda }\longrightarrow _{\mathcal{F}}\infty \text{, }x_{\lambda
}\longrightarrow x\text{ and }t_{\lambda }x_{\lambda }\longrightarrow y%
\end{array}%
\right\} .
\end{equation*}
\end{definition}

The following lemma is proved in \cite[Proposition 2.14]{BRS}.

\begin{lemma}
\label{propwJ} If $K\subset X$ is compact and $U\subset X$ is open, one has 
\begin{equation*}
\begin{array}{lll}
\omega \left( K,\mathcal{F}\right) \subset J\left( K,\mathcal{F}\right) & 
\text{and} & J\left( U,\mathcal{F}\right) \subset \omega \left( U,\mathcal{F}%
\right) \text{.}%
\end{array}%
\end{equation*}
\end{lemma}

We might assume the following additional hypotheses on the filter basis $%
\mathcal{F}$.

\begin{definition}
\label{hipH} The family $\mathcal{F}$ is said to satisfy:

\begin{enumerate}
\item Hypothesis $\mathrm{H}_{1}$ if for all $s\in \mathcal{S}$ and $A\in 
\mathcal{F}$ there exists $B\in \mathcal{F}$ such that $sB\subset A$.

\item Hypothesis $\mathrm{H}_{2}$ if for all $s\in \mathcal{S}$ and $A\in 
\mathcal{F}$ there exists $B\in \mathcal{F}$ such that $Bs\subset A$.

\item Hypothesis $\mathrm{H}_{3}$ if for all $s\in \mathcal{S}$ and $A\in 
\mathcal{F}$ there exists $B\in \mathcal{F}$ such that $B\subset As$.

\item Hypothesis $\mathrm{H}_{4}$ if for all $s\in \mathcal{S}$ and $A\in 
\mathcal{F}$ there exists $B\in \mathcal{F}$ such that $B\subset sA$.
\end{enumerate}
\end{definition}

We refer to papers \cite{smj}, \cite{smj2}, \cite{smj3} for examples of
usual families satisfying these hypotheses.

\begin{remark}
\label{Forwardinv} Let $\mathcal{F}$ be a family of subsets of $S$
satisfying Hypothesis $\mathrm{H}_{1}$. Then $\omega \left( Y,\mathcal{F}%
\right) $ is forward invariant if it is nonempty (\cite[Proposition 2.10]%
{smj}).
\end{remark}

\begin{definition}
The subset $Y\subset X$ \textbf{attracts} the subset $Z\subset X$ on the
direction of the filter basis $\mathcal{F}$ if for each $\mathcal{U}\in 
\mathcal{O}$ there is $A\in \mathcal{F}$ such that $AZ\subset \mathrm{St}%
\left[ Y,\mathcal{U}\right] $; $Y$ \textbf{absorbs} $Z$ on the direction of $%
\mathcal{F}$ if there is $A\in \mathcal{F}$ such that $AZ\subset Y$.
\end{definition}

We can describe attraction by means of $\mathcal{F}$-divergent nets.

\begin{proposition}
\label{P3}The subset $Y\subset X$ attracts the subset $Z\subset X$ on the
direction of the filter basis $\mathcal{F}$ if and only if $\rho _{Y}\left(
t_{\lambda }Z\right) \rightarrow \mathcal{O}$ for every net $t_{\lambda
}\rightarrow _{\mathcal{F}}\infty $.
\end{proposition}

\begin{proof}
Suppose that $Y$ attracts $Z$ and take a net $t_{\lambda }\rightarrow _{%
\mathcal{F}}\infty $. For a given $\mathcal{U}\in \mathcal{O}$ there is $%
A\in \mathcal{F}$ such that $AZ\subset \mathrm{St}\left[ Y,\mathcal{U}\right]
$. For this $A$, there is $\lambda _{0}$ such that $\lambda \geq \lambda
_{0} $ implies $t_{\lambda }\in A$, hence $t_{\lambda }Z\subset \mathrm{St}%
\left[ Y,\mathcal{U}\right] $ for all $\lambda \geq \lambda _{0}$, that is, $%
\mathcal{U}\in \rho _{Y}\left( t_{\lambda }Z\right) $ for all $\lambda \geq
\lambda _{0}$. Therefore $\rho _{Y}\left( t_{\lambda }Z\right) \rightarrow 
\mathcal{O}$. As to the converse, assume that $\rho _{Y}\left( t_{\lambda
}Z\right) \rightarrow \mathcal{O}$ for every net $t_{\lambda }\rightarrow _{%
\mathcal{F}}\infty $ and suppose by contradiction that $Y$ does not attracts 
$Z$. Then there is some $\mathcal{U}\in \mathcal{O}$ such that $AZ\nsubseteq 
\mathrm{St}\left[ Y,\mathcal{U}\right] $ for all $A\in \mathcal{F}$. Hence,
for each $A\in \mathcal{F}$, there is $t_{A}\in A$ such that $%
t_{A}Z\nsubseteq \mathrm{St}\left[ Y,\mathcal{U}\right] $. As $%
t_{A}\rightarrow _{\mathcal{F}}\infty $, we have $\rho _{Y}\left(
t_{A}Z\right) \rightarrow \mathcal{O}$, by hypothesis. This means there is $%
A_{0}$ such that $A\subset A_{0}$ implies $\mathcal{U}\in \rho _{Y}\left(
t_{A}Z\right) $, and therefore $t_{A}Z\subset \mathrm{St}\left[ Y,\mathcal{U}%
\right] $ for all $A\subset A_{0}$, which is a contradiction.
\end{proof}

We now introduce the notions of eventually compact, eventually bounded,
bounded dissipative, asymptotically compact, and $\omega $-limit compact
semigroup actions. These concepts were introduced in \cite{Stephanie}.

\begin{definition}
The semigroup action $\left( S,X\right) $ is called:

\begin{enumerate}
\item $\mathcal{F}$\textbf{-eventually bounded} if for each bounded set $%
Y\subset X$ there is $A\in \mathcal{F}$ such that $AY$ is a bounded set of $%
X $.

\item $\mathcal{F}$\textbf{-bounded dissipative} if there is a bounded
subset $D$ of $X$ that absorbs every bounded subset of $X$.

\item $\mathcal{F}$\textbf{-point dissipative} if there is a bounded subset $%
D\subset X$ that absorbs every point $x\in X$.

\item $\mathcal{F}$\textbf{-asymptotically compact} if for all bounded net $%
\left( x_{\lambda }\right) $ in $X$ and all net $\left( t_{\lambda }\right) $
in $S$ with $t_{\lambda }\longrightarrow _{\mathcal{F}}\infty $, the net $%
\left( t_{\lambda }x_{\lambda }\right) $ has convergent subnet.

\item $\mathcal{F}$\textbf{-limit compact} if for every bounded set $%
B\subset X$ and any $\mathcal{U}\in \mathcal{O}$ there is $A\in \mathcal{F}$
such that $\mathcal{U}\in \beta \left( AB\right) $.
\end{enumerate}
\end{definition}

We can characterize bounded dissipative semigroup actions as follows.

\begin{proposition}
\label{BoundedDissip} The semigroup action $\left( S,X\right) $ is $\mathcal{%
F}$-bounded dissipative if and only if there is a bounded subset $D\subset X$
which attracts every bounded subset of $X$ with respect to the family $%
\mathcal{F}$.
\end{proposition}

\begin{proof}
Suppose that $\left( S,X\right) $ is $\mathcal{F}$-bounded dissipative. Then
there is a bounded set $D\subset X$ that absorbs every bounded subset of $X$%
. Let $Y$ be a bounded subset of $X$. Then there is $A\in \mathcal{F}$ such
that $AY\subset D$. Hence, $AY\subset \mathrm{St}\left[ D,\mathcal{U}\right] 
$ for every $\mathcal{U}\in \mathcal{O}$, and therefore $D$ attracts $Y$. As
to the converse, suppose the existence of a bounded set $Y\subset X$ which
attracts every bounded subset of $X$. Choose $\mathcal{U}\in \mathcal{O}$
and define $D=\mathrm{St}\left[ Y,\mathcal{U}\right] $. Given a bounded
subset $Z$ of $X$, there is $A\in \mathcal{F}$ such that $AZ\subset \mathrm{%
St}\left[ Y,\mathcal{U}\right] =D$, that is, $D$ absorbs $Z$. Therefore, the
semigroup action is $\mathcal{F}$-bounded dissipative.
\end{proof}

We have the following relation among eventually compact, eventually bounded,
and asymptotically compact semigroup actions.

\begin{proposition}
\label{1.1} Assume that $\mathcal{F}$ satisfies Hypothesis $\mathrm{H}_{4}$.
If the semigroup action $\left( S,X\right) $ is both eventually compact and $%
\mathcal{F}$-eventually bounded then $\left( S,X\right) $ is $\mathcal{F}$%
-asymptotically compact.
\end{proposition}

\begin{proof}
Let $\left( x_{\lambda }\right) _{\lambda \in \Lambda }$ be a bounded net in 
$X$ and $\left( t_{\lambda }\right) _{\lambda \in \Lambda }$ a net in $S$
with $t_{\lambda }\longrightarrow _{\mathcal{F}}\infty $. Since $\left(
S,X\right) $ is $\mathcal{F}$-eventually bounded and the set $Y=\left\{
x_{\lambda };\text{ }\lambda \in \Lambda \right\} $ is bounded, there is $%
A\in \mathcal{F}$ such that $AY$ is bounded. As $\left( S,X\right) $ is
eventually compact, there is $t\in S$ such that $\mu _{t}:X\rightarrow X$ is
a compact map. Hypothesis $\mathrm{H}_{4}$ yields the existence of $B\in 
\mathcal{F}$ such that $B\subset tA$. Since $t_{\lambda }\longrightarrow _{%
\mathcal{F}}\infty $, there is $\lambda _{0}\in \Lambda $ such that $%
t_{\lambda }\in B$ for all $\lambda \geq \lambda _{0}$. Hence, for each $%
\lambda \geq \lambda _{0}$, we can write $t_{\lambda }=ts_{\lambda }$, with $%
s_{\lambda }\in A$. Now, define $Z=\left\{ s_{\lambda }x_{\lambda };\text{ }%
\lambda \geq \lambda _{0}\right\} $. As $Z\subset AY$, it follows that $Z$
is bounded. Then the set%
\begin{equation*}
\mu _{t}\left( Z\right) =\left\{ ts_{\lambda }x_{\lambda };\text{ }\lambda
\geq \lambda _{0}\right\} =\left\{ t_{\lambda }x_{\lambda };\text{ }\lambda
\geq \lambda _{0}\right\}
\end{equation*}%
is relatively compact, which implies the net $\left( t_{\lambda }x_{\lambda
}\right) $ has convergent subnet. Therefore $\left( S,X\right) $ is $%
\mathcal{F}$-asymptotically compact.
\end{proof}

Asymptotically compact semigroup actions have relevant properties, as stated
in the following sequence of theorems.

\begin{proposition}
\label{limitset2} Assume that the semigroup action $\left( S,X\right) $ is $%
\mathcal{F}$-asymptotically compact, and let $Y\subset X$ be a nonempty
bounded subset. Then the limit set $\omega \left( Y,\mathcal{F}\right) $ is
nonempty, compact, and attracts $Y$. Furthermore, $\omega \left( Y,\mathcal{F%
}\right) $ is the smallest closed subset of $X$ that attracts $Y$.
\end{proposition}

\begin{proof}
By asymptotic compactness, every net $\left( t_{\lambda }y_{\lambda }\right) 
$, with $t_{\lambda }\rightarrow _{\mathcal{F}}\infty $ and $y_{\lambda }\in
Y$, has convergent subnet. Hence $\omega \left( Y,\mathcal{F}\right) $ is
nonempty. To show that $\omega \left( Y,\mathcal{F}\right) $ is compact,
take an arbitrary net $\left( x_{\lambda }\right) _{\lambda \in \Lambda }$
in $\omega \left( Y,\mathcal{F}\right) $. For each $\lambda \in \Lambda $,
there is a net $\left( t_{\left( A,\mathcal{U}\right) }^{\lambda }y_{\left(
A,\mathcal{U}\right) }^{\lambda }\right) _{\left( A,\mathcal{U}\right) \in 
\mathcal{F}\times \mathcal{O}}$ such that $t_{\left( A,\mathcal{U}\right)
}^{\lambda }\in A$ for every $A\in \mathcal{F}$, $y_{\left( A,\mathcal{U}%
\right) }^{\lambda }\in Y$, and $t_{\left( A,\mathcal{U}\right) }^{\lambda
}y_{\left( A,\mathcal{U}\right) }^{\lambda }\rightarrow x_{\lambda }$.
Denote $\sigma =\left( A,\mathcal{U}\right) $ and $\Sigma =\mathcal{F}\times 
\mathcal{O}$. Consider $\Lambda \times \Sigma $ directed by the product
direction. Fix $\lambda _{0}\in \Lambda $ and $\mathcal{U}_{0}\in \mathcal{O}
$. For a given $A\in \mathcal{F}$, $\left( \lambda ,\left( B,\mathcal{U}%
\right) \right) \geqslant \left( \lambda _{0},\left( B,\mathcal{U}%
_{0}\right) \right) $ means $B\subset A$, and then $t_{\left( B,\mathcal{U}%
\right) }^{\lambda }\in B\subset A$. Hence the net $\left( t_{\sigma
}^{\lambda }\right) _{\left( \lambda ,\sigma \right) \in \Lambda \times
\Sigma }$ diverges with respect to $\mathcal{F}$, that is, $t_{\sigma
}^{\lambda }\rightarrow _{\mathcal{F}}\infty $. Since $y_{\sigma }^{\lambda
}\in Y$, it follows that the net $\left( t_{\sigma }^{\lambda }y_{\sigma
}^{\lambda }\right) _{\left( \lambda ,\sigma \right) \in \Lambda \times
\Sigma }$ admits a convergent subnet. Thus we may assume that $\left(
t_{\sigma }^{\lambda }y_{\sigma }^{\lambda }\right) $ converges to some
point $x\in X$. This means that $x\in \omega \left( Y,\mathcal{F}\right) $.
We claim that a subnet of $\left( x_{\lambda }\right) _{\lambda \in \Lambda }
$ converges to $x$. In fact, for a given $\mathcal{U}\in \mathcal{O}$, take
a double-refinement $\mathcal{V}\in \mathcal{O}$ of $\mathcal{U}$. There is $%
\left( \lambda _{0},\sigma _{0}\right) $ such that $\left( \lambda ,\sigma
\right) \geqslant \left( \lambda _{0},\sigma _{0}\right) $ implies $%
t_{\sigma }^{\lambda }y_{\sigma }^{\lambda }\in \mathrm{St}\left[ x,\mathcal{%
V}\right] $. For each $\lambda \in \Lambda $, we can get $\lambda _{\mathcal{%
U}}\in \Lambda $ such that $\lambda _{\mathcal{U}}\geq \lambda $ and $%
\lambda _{\mathcal{U}}\geq \lambda _{0}$. For this $\lambda _{\mathcal{U}}$
we can take $\sigma _{\mathcal{U}}\in \Sigma $ such that $\sigma _{\mathcal{U%
}}\geq \sigma _{0}$ and $t_{\sigma _{\mathcal{U}}}^{\lambda _{\mathcal{U}%
}}y_{\sigma _{\mathcal{U}}}^{\lambda _{\mathcal{U}}}\in \mathrm{St}\left[
x_{\lambda _{\mathcal{U}}},\mathcal{V}\right] $. As $\left( \lambda _{%
\mathcal{U}},\sigma _{\mathcal{U}}\right) \geqslant \left( \lambda
_{0},\sigma _{0}\right) $, it follows that $t_{\sigma _{\mathcal{U}%
}}^{\lambda _{\mathcal{U}}}y_{\sigma _{\mathcal{U}}}^{\lambda _{\mathcal{U}%
}}\in \mathrm{St}\left[ x_{\lambda _{\mathcal{U}}},\mathcal{V}\right] \cap 
\mathrm{St}\left[ x,\mathcal{V}\right] $, and then $x_{\lambda _{\mathcal{U}%
}}\in \mathrm{St}\left[ x,\mathcal{U}\right] $, because $\mathcal{V}%
\leqslant \frac{1}{2}\mathcal{U}$. Hence the subnet $\left( x_{\lambda _{%
\mathcal{U}}}\right) _{\mathcal{U}\in \mathcal{O}}$ of $\left( x_{\lambda
}\right) $ converges to $x$, and therefore $\omega \left( Y,\mathcal{F}%
\right) $ is compact. We now show that $\omega \left( Y,\mathcal{F}\right) $
attracts $Y$. Suppose by contradiction that there is some $\mathcal{U}\in 
\mathcal{O}$ such that $AY\nsubseteq \mathrm{St}\left[ \omega \left( Y,%
\mathcal{F}\right) ,\mathcal{U}\right] $ for all $A\in \mathcal{F}$. Then,
for each $A\in \mathcal{F}$, there is $t_{A}\in A$ and $y_{A}\in Y$ such
that $t_{A}y_{A}\in X\setminus \mathrm{St}\left[ \omega \left( Y,\mathcal{F}%
\right) ,\mathcal{U}\right] $. As $t_{A}\longrightarrow _{\mathcal{F}}\infty 
$, the net $\left( t_{A}y_{A}\right) $ has a subnet $\left( t_{A_{\lambda
}}y_{A_{\lambda }}\right) $ that converges to some point $x\in X$. Hence $%
x\in \omega \left( Y,\mathcal{F}\right) $. By Proposition \ref{R5}, we have $%
\rho _{\omega \left( Y,\mathcal{F}\right) }\left( t_{A_{\lambda
}}y_{A_{\lambda }}\right) \rightarrow \mathcal{O}$, hence there is $\lambda
_{0}$ such that $\lambda \geq \lambda _{0}$ implies $\mathcal{U}\in \rho
_{\omega \left( Y,\mathcal{F}\right) }\left( t_{A_{\lambda }}y_{A_{\lambda
}}\right) $. Therefore $t_{A_{\lambda }}y_{A_{\lambda }}\in \mathrm{St}\left[
\omega \left( Y,\mathcal{F}\right) ,\mathcal{U}\right] $ for all $\lambda
\geq \lambda _{0}$, which is a contradiction. Thus $\omega \left( Y,\mathcal{%
F}\right) $ attracts $Y$. Finally, we show that $\omega \left( Y,\mathcal{F}%
\right) $ is the smallest closed subset of $X$ that attracts $Y$. Let $%
K\subset X$ be a closed set that attracts $Y$. Take $x\in \omega \left( Y,%
\mathcal{F}\right) $ and a net $t_{\lambda }y_{\lambda }\rightarrow x$ with $%
t_{\lambda }\rightarrow _{\mathcal{F}}\infty $ and $y_{\lambda }\in Y$. By
Proposition \ref{P3}, we have $\rho _{K}\left( t_{\lambda }Y\right)
\rightarrow \mathcal{O}$. As $\rho _{K}\left( t_{\lambda }y_{\lambda
}\right) \prec \rho _{K}\left( t_{\lambda }Y\right) $ for all $\lambda $, it
follows that $\rho _{K}\left( t_{\lambda }y_{\lambda }\right) \rightarrow 
\mathcal{O}$. By Proposition \ref{R5}, $x\in K$, and therefore $\omega
\left( Y,\mathcal{F}\right) \subset K$.
\end{proof}

\begin{proposition}
Assume that $X$ is Hausdorff, the semigroup action $\left( S,X\right) $ is $%
\mathcal{F}$-asymptotically compact and $\mathcal{F}$ satisfies both
Hypotheses $\mathrm{H}_{1}$ and $\mathrm{H}_{4}$. Then the limit set $\omega
\left( Y,\mathcal{F}\right) $ is invariant for every bounded subset $Y$ of $%
X $.
\end{proposition}

\begin{proof}
Let $Y\subset X$ be a bounded subset and $s\in S$. By Remark \ref{Forwardinv}%
, we have $s\omega \left( Y,\mathcal{F}\right) \subset \omega \left( Y,%
\mathcal{F}\right) $. On the other hand, let $x\in \omega \left( Y,\mathcal{F%
}\right) $. For a given $A\in \mathcal{F}$, there is $B\in \mathcal{F}$ such
that $B\subset sA$. Then we have%
\begin{equation*}
\emptyset \neq BY\cap \mathrm{St}\left[ x,\mathcal{U}\right] \subset sAY\cap 
\mathrm{St}\left[ x,\mathcal{U}\right]
\end{equation*}%
for all $\mathcal{U}\in \mathcal{O}$. Hence we can take $st_{\left( A,%
\mathcal{U}\right) }y_{\left( A,\mathcal{U}\right) }\in sAY\cap \mathrm{St}%
\left[ x,\mathcal{U}\right] $, with $t_{\left( A,\mathcal{U}\right) }\in A$
and $y_{\left( A,\mathcal{U}\right) }\in Y$. For a given $A_{0}\in \mathcal{F%
}$, $\left( A,\mathcal{V}\right) \geqslant \left( A_{0},\mathcal{U}\right) $
implies $t_{\left( A,\mathcal{V}\right) }\in A\subset A_{0}$. Hence, $%
t_{\left( A,\mathcal{U}\right) }\rightarrow _{\mathcal{F}}\infty $. Since $%
\left( S,X\right) $ is $\mathcal{F}$-asymptotically compact and $Y$ is
bounded, we can assume that $t_{\left( A,\mathcal{U}\right) }y_{\left( A,%
\mathcal{U}\right) }\longrightarrow y$ in $X$. Then $y\in \omega \left( Y,%
\mathcal{F}\right) $. Moreover, as $st_{\left( A,\mathcal{U}\right)
}y_{\left( A,\mathcal{U}\right) }\in \mathrm{St}\left[ x,\mathcal{U}\right] $
for every $\left( A,\mathcal{U}\right) \in \mathcal{F}\times \mathcal{O}$,
it follows that $st_{\left( A,\mathcal{U}\right) }y_{\left( A,\mathcal{U}%
\right) }\rightarrow x$. Since $st_{\left( A,\mathcal{U}\right) }y_{\left( A,%
\mathcal{U}\right) }\rightarrow sy$, we have $x=sy\in s\omega \left( Y,%
\mathcal{F}\right) $. Hence, $\omega \left( Y,\mathcal{F}\right) \subset
s\omega \left( Y,\mathcal{F}\right) $, and therefore $\omega \left( Y,%
\mathcal{F}\right) $ is invariant.
\end{proof}

\begin{proposition}
Assume that the semigroup action $\left( S,X\right) $ is $\mathcal{F}$%
-asymptotically compact and the set $AC$ is connected whenever $C\subset X$
is connected and $A\in \mathcal{F}$. If $Y\subset X$ is a bounded subset and
there is $C\supset Y$ connected which is attracted by $\omega \left( Y,%
\mathcal{F}\right) $, then $\omega \left( Y,\mathcal{F}\right) $ is
connected.
\end{proposition}

\begin{proof}
Suppose that $C\supset Y$ is connected and $\omega \left( Y,\mathcal{F}%
\right) $ attracts $C$. Suppose by contradiction that $\omega \left( Y,%
\mathcal{F}\right) $ is disconnected. Then $\omega \left( Y,\mathcal{F}%
\right) =K_{1}\cup K_{2}$, where $K_{1}$ and $K_{2}$ are two nonempty closed
sets with $K_{1}\cap K_{2}=\emptyset $. By Proposition \ref{limitset2}, $%
\omega \left( Y,\mathcal{F}\right) $ is compact, and then both $K_{1}$ and $%
K_{2}$ are compact. Hence, there is $\mathcal{U}\in \mathcal{O}$ such that $%
\mathrm{St}\left[ K_{1},\mathcal{U}\right] \cap \mathrm{St}\left[ K_{2},%
\mathcal{U}\right] =\emptyset $. Take $\mathcal{V}\in \mathcal{O}$ with $%
\mathcal{V}\leqslant \frac{1}{2}\mathcal{U}$. Since $\omega \left( Y,%
\mathcal{F}\right) $ attracts $C$, there is $A\in \mathcal{F}$ such that%
\begin{equation*}
AC\subset \mathrm{St}\left[ \omega \left( Y,\mathcal{F}\right) ,\mathcal{V}%
\right] =\mathrm{St}\left[ K_{1},\mathcal{V}\right] \cup \mathrm{St}\left[
K_{2},\mathcal{V}\right] .
\end{equation*}%
Since $AC$ is connected and $\mathrm{St}\left[ K_{1},\mathcal{V}\right] \cap 
\mathrm{St}\left[ K_{2},\mathcal{V}\right] =\emptyset $, we have either $%
AC\subset \mathrm{St}\left[ K_{1},\mathcal{V}\right] $ or $AC\subset \mathrm{%
St}\left[ K_{2},\mathcal{V}\right] $. Suppose that $AC\subset \mathrm{St}%
\left[ K_{1},\mathcal{V}\right] $. Then we have%
\begin{equation*}
K_{2}\subset \omega \left( Y,\mathcal{F}\right) \subset \mathrm{cls}\left(
AC\right) \subset \mathrm{cls}\left( \mathrm{St}\left[ K_{1},\mathcal{V}%
\right] \right) \subset \mathrm{St}\left[ K_{1},\mathcal{U}\right]
\end{equation*}%
which contradicts $\mathrm{St}\left[ K_{1},\mathcal{U}\right] \cap \mathrm{St%
}\left[ K_{2},\mathcal{U}\right] =\emptyset $. Therefore $\omega \left( Y,%
\mathcal{F}\right) $ is connected.
\end{proof}

Limit compact semigroup actions seem quite a bit like asymptotically compact
semigroup actions, as the following.

\begin{theorem}
If $\left( S,X\right) $ is $\mathcal{F}$-asymptotically compact then it is $%
\mathcal{F}$-limit compact. The converse holds if $X$ is complete.
\end{theorem}

\begin{proof}
Suppose that $\left( S,X\right) $ is $\mathcal{F}$-asymptotically compact
and let $Y\subset X$ be a nonempty bounded set. By Proposition \ref%
{limitset2}, $\omega \left( Y,\mathcal{F}\right) $ is nonempty, compact, and
attracts $Y$. For a given $\mathcal{V}\in \mathcal{O}$, the open covering $%
\omega \left( Y,\mathcal{F}\right) \subset \underset{x\in \omega \left( Y,%
\mathcal{F}\right) }{\bigcup }\mathrm{St}\left[ x,\mathcal{V}\right] $
admits a finite subcovering $\omega \left( Y,\mathcal{F}\right) \subset 
\overset{n}{\underset{i=1}{\bigcup }}\mathrm{St}\left[ x_{i},\mathcal{V}%
\right] $. Take $\mathcal{U}\in \mathcal{O}$ such that $\mathrm{St}\left[
\omega \left( Y,\mathcal{F}\right) ,\mathcal{U}\right] \subset \overset{n}{%
\underset{i=1}{\bigcup }}\mathrm{St}\left[ x_{i},\mathcal{V}\right] $. Since 
$\omega \left( Y,\mathcal{F}\right) $ attracts $Y$, there is $A\in \mathcal{F%
}$ such that $AY\subset \mathrm{St}\left[ \omega \left( Y,\mathcal{F}\right)
,\mathcal{U}\right] \subset \overset{n}{\underset{i=1}{\bigcup }}\mathrm{St}%
\left[ x_{i},\mathcal{V}\right] $, hence $\mathcal{V}\in \beta \left(
AY\right) $. Thus $\left( S,X\right) $ is $\mathcal{F}$-limit compact. For
the converse, suppose that $X$ is a complete admissible space and $\left(
S,X\right) $ is $\mathcal{F}$-limit compact. Let $\left( x_{\lambda }\right)
_{\lambda \in \Lambda }$ be a bounded net in $X$ and $t_{\lambda
}\longrightarrow _{\mathcal{F}}\infty $ in $S$. Set $Y=\left\{ x_{\lambda
}:\lambda \in \Lambda \right\} $. For each $\mathcal{U}\in \mathcal{O}$, we
can use the limit compactness and Proposition \ref{P9} to find $A_{\mathcal{U%
}}\in \mathcal{F}$ such that $\mathcal{U}\in \beta \left( \mathrm{cls}\left(
A_{\mathcal{U}}Y\right) \right) $. For each $\lambda \in \Lambda $, set $%
F_{\lambda }=\mathrm{cls}\left\{ t_{\lambda ^{\prime }}x_{\lambda ^{\prime
}}:\lambda ^{\prime }\geq \lambda \right\} $. For a given $\mathcal{U}\in 
\mathcal{O}$, there is $\lambda _{\mathcal{U}}$ such that $\lambda \geq
\lambda _{\mathcal{U}}$ implies $t_{\lambda }\in A_{\mathcal{U}}$. Hence, $%
F_{\lambda }\subset \mathrm{cls}\left( A_{\mathcal{U}}Y\right) $ for all $%
\lambda \geq \lambda _{\mathcal{U}}$, and therefore $\beta \left( F_{\lambda
}\right) \prec \beta \left( \mathrm{cls}\left( A_{\mathcal{U}}Y\right)
\right) \prec \left\{ \mathcal{U}\right\} $ for all $\lambda \geq \lambda _{%
\mathcal{U}}$. This means that $\gamma \left( F_{\lambda }\right)
\rightarrow \mathcal{O}$. Moreover, if $\lambda ^{\prime }\geq \lambda $
then $F_{\lambda ^{\prime }}\subset F_{\lambda }$, and then $\left(
F_{\lambda }\right) _{\lambda \in \Lambda }$ is a decreasing net of nonempty
closed sets. By Theorem \ref{CantorKuratowski}, the intersection $\underset{%
\lambda \in \Lambda }{\bigcap }F_{\lambda }$ is nonempty, and then we can
get $x\in \underset{\lambda \in \Lambda }{\bigcap }F_{\lambda }$. Now, for
each $\mathcal{U}\in \mathcal{O}$ and $\lambda \in \Lambda $, there is $%
\lambda _{\mathcal{U}}^{\prime }\geq \lambda $ such that $t_{\lambda _{%
\mathcal{U}}^{\prime }}x_{\lambda _{\mathcal{U}}^{\prime }}\in \mathrm{St}%
\left[ x,\mathcal{U}\right] $, because $x\in F_{\lambda }$. Hence $\left(
t_{\lambda _{\mathcal{U}}^{\prime }}x_{\lambda _{\mathcal{U}}^{\prime
}}\right) $ is a subnet of $\left( t_{\lambda }x_{\lambda }\right) $ such
that $t_{\lambda _{\mathcal{U}}^{\prime }}x_{\lambda _{\mathcal{U}}^{\prime
}}\rightarrow x$. Therefore $\left( S,X\right) $ is $\mathcal{F}$%
-asymptotically compact.
\end{proof}

\section{Global Attractors}

In this section we discuss necessary and sufficient conditions for the
existence of global attractor for semigroup actions on completely regular
spaces. Throughout, there is a fixed semigroup action $\left( S,X,\mu
\right) $ with $X$ an admissible space endowed with an admissible family of
open coverings $\mathcal{O}$. We refer to \cite{Stephanie} and \cite{So4}
for previous discussion and illustrating examples of global attractors.

\begin{definition}
\label{GlobalAttractor} The subset $\mathcal{A}$ of $X$ is called \textbf{%
global }$\mathcal{F}$\textbf{-attractor} if $\mathcal{A}$ is nonempty,
closed, compact, invariant, and it attracts every bounded subset of $X$; it
is called \textbf{global uniform }$\mathcal{F}$\textbf{-attractor} if it is
a compact invariant set, $J\left( x,\mathcal{F}\right) \neq \emptyset $, for
every $x\in X$, and $J\left( X,\mathcal{F}\right) \subset \mathcal{A}$.
\end{definition}

There is at most one global $\mathcal{F}$-attractor for the semigroup action.

\begin{theorem}
\label{1.2} If the semigroup action $\left( S,X\right) $ admits a global $%
\mathcal{F}$-attractor $\mathcal{A}$ then it is unique and coincides with
the reunion of all bounded invariant subsets of $X$.
\end{theorem}

\begin{proof}
Let $\mathcal{A}$ and $\mathcal{A}^{\prime }$ be two global $\mathcal{F}$%
-attractors. Since $\mathcal{A}^{\prime }$ is compact, it is bounded, by
Remark \ref{R1}. Hence, $\mathcal{A}$ attracts $\mathcal{A}^{\prime }$, that
is, $\rho _{\mathcal{A}}\left( t_{\lambda }\mathcal{A}^{\prime }\right)
\rightarrow \mathcal{O}$ for every net $t_{\lambda }\rightarrow _{\mathcal{F}%
}\infty $. Since $\mathcal{A}^{\prime }$ is invariant, it follows that $\rho
_{\mathcal{A}}\left( t_{\lambda }\mathcal{A}^{\prime }\right) =\rho _{%
\mathcal{A}}\left( \mathcal{A}^{\prime }\right) $ and hence $\rho _{\mathcal{%
A}}\left( \mathcal{A}^{\prime }\right) =\mathcal{O}$. Therefore $\mathcal{A}%
^{\prime }\subset \mathrm{cls}\left( \mathcal{A}\right) =\mathcal{A}$. In
the same way we show that $\mathcal{A}\subset \mathcal{A}^{\prime }$. Thus $%
\mathcal{A}=\mathcal{A}^{\prime }$. Now, let $\left\{ B_{\sigma }\right\}
_{\sigma \in \Sigma }$ be the set of all bounded invariant subsets of $X$.
Since $\mathcal{A}$ is bounded, we have $\mathcal{A}\subset \bigcup_{\sigma
\in \Sigma }B_{\sigma }$. On the other hand, $\mathcal{A}$ attracts every $%
B_{\sigma }$, that is, $\rho _{\mathcal{A}}\left( t_{\lambda }B_{\sigma
}\right) \rightarrow \mathcal{O}$ for every net $t_{\lambda }\rightarrow _{%
\mathcal{F}}\infty $. As $B_{\sigma }$ is invariant, it follows that $%
B_{\sigma }\subset \mathcal{A}$, and therefore $\bigcup_{\sigma \in \Sigma
}B_{\sigma }\subset \mathcal{A}$.
\end{proof}

Similarly, there is at most one global uniform $\mathcal{F}$-attractor for $%
\left( S,X\right) $ (\cite[Theorem 2.4]{So4}).

The following theorem presents a criteria for the existence of the global
attractor together with a characterization of it.

\begin{theorem}
\label{1.3} If the semigroup action $\left( S,X\right) $ has a global $%
\mathcal{F}$-attractor $\mathcal{A}$ then $\left( S,X\right) $ is $\mathcal{F%
}$-eventually bounded, $\mathcal{F}$-bounded dissipative, and $\mathcal{F}$%
-asymptotically compact, and 
\begin{eqnarray*}
\mathcal{A} &=&\underset{Y\in \mathfrak{B}}{\bigcup }\omega \left( Y,%
\mathcal{F}\right) \\
&=&\left\{ x\in X:\text{there is a bounded net }\left( x_{\lambda }\right) 
\text{ in }X\text{ and }t_{\lambda }\rightarrow _{\mathcal{F}}\infty \text{
such that }t_{\lambda }x_{\lambda }\rightarrow x\right\}
\end{eqnarray*}%
where $\mathfrak{B}$ is the collection of all bounded subsets of $X$. The
converse holds if the $\omega $-limit sets of bounded subsets of $X$ are
invariant.
\end{theorem}

\begin{proof}
By Proposition \ref{BoundedDissip}, $\left( S,X\right) $ is $\mathcal{F}$%
-bounded dissipative. Let $Y$ be a bounded subset of $X$. For $\mathcal{U}%
\in \mathcal{O}$, there is $A\in \mathcal{F}$ such that $AY\subset \mathrm{St%
}\left[ \mathcal{A},\mathcal{U}\right] $. Since $\mathcal{A}$ is bounded, $%
\mathrm{St}\left[ \mathcal{A},\mathcal{U}\right] $ is also bounded, by
Remark \ref{R2}. Hence $AY$ is bounded, and therefore $\left( S,X\right) $
is $\mathcal{F}$-eventually bounded. Now, let $\left( x_{\lambda }\right)
_{\lambda \in \Lambda }$ be a bounded net in $X$ and $t_{\lambda
}\rightarrow _{\mathcal{F}}\infty $. By taking the bounded set $Y=\left\{
x_{\lambda }\in X:\lambda \in \Lambda \right\} $, we have $\rho _{\mathcal{A}%
}\left( t_{\lambda }Y\right) \rightarrow \mathcal{O}$. Hence $\rho _{%
\mathcal{A}}\left( t_{\lambda }y_{\lambda }\right) \rightarrow \mathcal{O}$.
For each $\mathcal{U}\in \mathcal{O}$ and $\lambda \in \Lambda $, there is $%
\lambda _{\mathcal{U}}\geq \lambda $ such that $\mathcal{U}\in \rho _{%
\mathcal{A}}\left( t_{\lambda _{\mathcal{U}}}y_{\lambda _{\mathcal{U}%
}}\right) $, and hence there is $a_{\lambda _{\mathcal{U}}}\in \mathcal{A}$
such that $\mathcal{U}\in \rho \left( t_{\lambda _{\mathcal{U}}}y_{\lambda _{%
\mathcal{U}}},a_{\lambda _{\mathcal{U}}}\right) $. By the compactness of $%
\mathcal{A}$, we may assume that $a_{\lambda _{\mathcal{U}}}$ converges to
some point $a\in \mathcal{A}$. Now, for a given $\mathcal{W}\in \mathcal{O}$%
, take $\mathcal{V}\in \mathcal{O}$ with $\mathcal{V}\leqslant \frac{1}{2}%
\mathcal{W}$. Since $a_{\lambda _{\mathcal{U}}}\rightarrow a$, there is $%
\mathcal{U}_{0}\in \mathcal{O}$ such that $\mathcal{U}\leqslant \mathcal{U}%
_{0}$ implies $a_{\lambda _{\mathcal{U}}}\in \mathrm{St}\left[ a,\mathcal{V}%
\right] $. If $\mathcal{U}\leqslant \mathcal{U}_{0}$ and $\mathcal{U}%
\leqslant \mathcal{V}$, we have $a_{\lambda _{\mathcal{U}}}\in \mathrm{St}%
\left[ a,\mathcal{V}\right] $ and $a_{\lambda _{\mathcal{U}}}\in \mathrm{St}%
\left[ t_{\lambda _{\mathcal{U}}}y_{\lambda _{\mathcal{U}}},\mathcal{U}%
\right] \subset \mathrm{St}\left[ t_{\lambda _{\mathcal{U}}}y_{\lambda _{%
\mathcal{U}}},\mathcal{V}\right] $, hence $t_{\lambda _{\mathcal{U}%
}}y_{\lambda _{\mathcal{U}}}\in \mathrm{St}\left[ a,\mathcal{W}\right] $,
because $\mathcal{V}\leqslant \frac{1}{2}\mathcal{W}$. It follows that the
subnet $\left( t_{\lambda _{\mathcal{U}}}y_{\lambda _{\mathcal{U}}}\right) _{%
\mathcal{U}\in \mathcal{O}}$ of $\left( t_{\lambda }y_{\lambda }\right)
_{\lambda \in \Lambda }$ converges to $a$, and therefore $\left( S,X\right) $
is $\mathcal{F}$-asymptotically compact. Finally, we show that $\mathcal{A}=%
\underset{Y\in \mathfrak{B}}{\bigcup }\omega \left( Y,\mathcal{F}\right) $.
Indeed, if $Y\in \mathfrak{B}$ then $\omega \left( Y,\mathcal{F}\right) $ is
the smallest closed subset of $X$ that attracts $Y$, by Proposition \ref%
{limitset2}. Hence $\omega \left( Y,\mathcal{F}\right) \subset \mathcal{A}$
for every $Y\in \mathfrak{B}$. Since $\mathcal{A}$ is invariant and closed,
we have $\mathcal{A}=\omega \left( \mathcal{A},\mathcal{F}\right) $. As $%
\mathcal{A}\in \mathfrak{B}$, $\mathcal{A}=\underset{Y\in \mathfrak{B}}{%
\bigcup }\omega \left( Y,\mathcal{F}\right) $. Conversely, assume that $%
\omega \left( Y,\mathcal{F}\right) $ is invariant, for every bounded subset $%
Y\subset X$, and $\left( S,X\right) $ is $\mathcal{F}$-asymptotically
compact and $\mathcal{F}$-bounded dissipative (and $\mathcal{F}$-eventually
bounded). Let $\mathfrak{B}$ be the collection of all nonempty bounded
subsets of $X$ and define $\mathcal{A}=\underset{Y\in \mathfrak{B}}{\bigcup }%
\omega \left( Y,\mathcal{F}\right) $. By Proposition \ref{limitset2}, $%
\omega \left( Y,\mathcal{F}\right) $ is nonempty, compact, and attracts $Y$
for every $Y\in \mathfrak{B}$. Hence $\mathcal{A}$ is nonempty and attracts
every bounded subset of $X$. Moreover, $\mathcal{A}$ is invariant, because
each $\omega \left( Y,\mathcal{F}\right) $ is invariant by hypothesis. Since 
$\left( S,X\right) $ is $\mathcal{F}$-bounded dissipative, there is a
bounded subset $D\subset X$ which absorbs $Y$ whenever $Y\in \mathfrak{B}$.
It follows that $\omega \left( Y,\mathcal{F}\right) \subset \mathrm{cls}%
\left( D\right) $ for all $Y\in \mathfrak{B}$. Hence $\mathcal{A}\subset 
\mathrm{cls}\left( D\right) $ and $\omega \left( \mathcal{A},\mathcal{F}%
\right) \subset \omega \left( D,\mathcal{F}\right) $. Since $\mathcal{A}$ is
invariant, we have $\mathcal{A}=\omega \left( \mathcal{A},\mathcal{F}\right)
\subset \omega \left( D,\mathcal{F}\right) $. On the other hand, as $D$ is
bounded, $\omega \left( D,\mathcal{F}\right) \subset \mathcal{A}$, and
therefore $\mathcal{A}=\omega \left( D,\mathcal{F}\right) $. By Proposition %
\ref{limitset2} again, $\mathcal{A}$ is compact. Thus $\mathcal{A}$ is the
global $\mathcal{F}$-attractor for $\left( S,X\right) $.\bigskip
\end{proof}

By considering point dissipative semigroup actions, Theorem \ref{1.3} yields
the following criteria for the existence of the global attractor.

\begin{theorem}
Assume that the limit set $\omega \left( Y,\mathcal{F}\right) $ is invariant
for every bounded subset $Y$ of $X$, the filter basis $\mathcal{F}$
satisfies both Hypotheses $\mathrm{H}_{3}$ and $\mathrm{H}_{4}$, and the
semigroup action $\left( S,X\right) $ is eventually compact, $\mathcal{F}$%
-eventually bounded, and $\mathcal{F}$-point dissipative. Then $\left(
S,X\right) $ has a global $\mathcal{F}$-attractor.
\end{theorem}

\begin{proof}
By Proposition \ref{1.1}, $\left( S,X\right) $ is $\mathcal{F}$%
-asymptotically compact. Since $\left( S,X\right) $ is $\mathcal{F}$-point
dissipative, there is a bounded subset $D_{0}$ of $X$ that absorbs every
point $x\in X$. For a given $\mathcal{U}\in \mathcal{O}$, we define the
bounded set 
\begin{equation*}
D_{1}=\mathrm{St}\left[ D_{0},\mathcal{U}\right] .
\end{equation*}%
As $\left( S,X\right) $ is $\mathcal{F}$-eventually bounded, there is $A\in 
\mathcal{F}$ such that the set $D=AD_{1}$ is bounded. This set $D$ absorbs
every bounded subset $Y$ of $X$. In fact, since the action is eventually
compact, there is $t\in S$ such that $K=\mathrm{cls}\left( tY\right) $ is
compact. For each $x\in K$, there is $A_{x}\in \mathcal{F}$ such that%
\begin{equation*}
A_{x}x\subset D_{0}\subset D_{1}.
\end{equation*}%
Hence, for $t_{x}\in A_{x}$, there is $\mathcal{U}_{x}\in \mathcal{O}$ such
that $t_{x}\mathrm{St}\left[ x,\mathcal{U}_{x}\right] \subset D_{1}$. By
Hypothesis $\mathrm{H}_{3}$, there is $A_{x}^{\prime }\in \mathcal{F}$ such
that $A_{x}^{\prime }\subset At_{x}$. Then we have%
\begin{equation*}
A_{x}^{\prime }\mathrm{St}\left[ x,\mathcal{U}_{x}\right] \subset At_{x}%
\mathrm{St}\left[ x,\mathcal{U}_{x}\right] \subset AD_{1}=D.
\end{equation*}%
Now we consider the open covering $K\subset \bigcup_{x\in K}\mathrm{St}\left[
x,\mathcal{U}_{x}\right] $. Since $K$ is compact, there are points $%
x_{1},...,x_{n}\in K$ such that $K\subset \bigcup_{j=1}^{n}\mathrm{St}\left[
x_{j},\mathcal{U}_{x_{j}}\right] $. Take $B\in \mathcal{F}$ such that $%
B\subset \bigcap_{j=1}^{n}A_{x_{j}}^{\prime }$. Then%
\begin{equation*}
BK\subset B\left( \bigcup_{j=1}^{n}\mathrm{St}\left[ x_{j},\mathcal{U}%
_{x_{j}}\right] \right) =\bigcup_{j=1}^{n}B\mathrm{St}\left[ x_{j},\mathcal{U%
}_{x_{j}}\right] \subset \bigcup_{j=1}^{n}A_{x_{j}}^{\prime }\mathrm{St}%
\left[ x_{j},\mathcal{U}_{x_{j}}\right] \subset D,
\end{equation*}%
that is, $B\mathrm{cls}\left( tY\right) \subset D$. Finally, take $B^{\prime
}\in \mathcal{F}$ such that $B^{\prime }\subset Bt$. It follows that 
\begin{equation*}
B^{\prime }Y\subset BtY\subset D,
\end{equation*}%
and therefore $D$ absorbs $Y$. Thus, the semigroup action $\left( S,X\right) 
$ is $\mathcal{F}$-bounded dissipative. By Proposition \ref{1.1}, $\left(
S,X\right) $ is also $\mathcal{F}$-asymptotically compact. The proof follows
by Theorem \ref{1.3}.
\end{proof}

We now relate global attractor to global uniform attractor.

\begin{theorem}
\label{GlobalUniform} Assume that $\mathcal{A}\subset X$ is a nonempty set
that is closed, compact, and invariant. If $\mathcal{A}$ is the global $%
\mathcal{F}$-attractor for $\left( S,X\right) $ then $\mathcal{A}$ is the
global uniform $\mathcal{F}$-attractor for $\left( S,X\right) $. The
converse holds if $\mathcal{F}$ satisfies Hypothesis $\mathrm{H}_{3}$ and $%
\left( S,X\right) $ is eventually compact and $\mathcal{F}$-asymptotically
compact.
\end{theorem}

\begin{proof}
Suppose that $\mathcal{A}$ is the global $\mathcal{F}$-attractor for $\left(
S,X\right) $ and let $x\in X$. By Theorem \ref{1.3}, $\left( S,X\right) $ is 
$\mathcal{F}$-asymptotically compact, hence $\omega \left( x,\mathcal{F}%
\right) \neq \emptyset $, by Proposition \ref{limitset2}. Since $\omega
\left( x,\mathcal{F}\right) \subset J\left( x,\mathcal{F}\right) $ we have $%
J\left( x,\mathcal{F}\right) \not=\emptyset $. Take any $\mathcal{U}\in 
\mathcal{O}$. Since $\mathrm{St}\left[ x,\mathcal{U}\right] $ is a bounded
set, we have $\omega \left( \mathrm{St}\left[ x,\mathcal{U}\right] ,\mathcal{%
F}\right) \subset \mathcal{A}$, by Theorem \ref{1.3}. By Lemma \ref{propwJ},
it follows that $J\left( \mathrm{St}\left[ x,\mathcal{U}\right] ,\mathcal{F}%
\right) \subset \omega \left( \mathrm{St}\left[ x,\mathcal{U}\right] ,%
\mathcal{F}\right) \subset \mathcal{A}$. Hence $J\left( X,\mathcal{F}\right)
\subset \mathcal{A}$, and therefore $\mathcal{A}$ is the global uniform $%
\mathcal{F}$-attractor for $\left( S,X\right) $. For the converse, suppose
that $\mathcal{F}$ satisfies Hypothesis $\mathrm{H}_{3}$, $\left( S,X\right) 
$ is eventually compact and $\mathcal{F}$-asymptotically compact, $J\left( x,%
\mathcal{F}\right) \not=\emptyset $ for all $x\in X$, and $J\left( X,%
\mathcal{F}\right) \subset \mathcal{A}$. Let $Y\subset X$ be a bounded
subset and take $t\in S$ such that $\mathrm{cls}\left( tY\right) $ is
compact. By Lemma \ref{propwJ}, $\omega \left( \mathrm{cls}\left( tY\right) ,%
\mathcal{F}\right) \subset J\left( \mathrm{cls}\left( tY\right) ,\mathcal{F}%
\right) $, and then $\omega \left( \mathrm{cls}\left( tY\right) ,\mathcal{F}%
\right) \subset \mathcal{A}$. By Proposition \ref{limitset2}, $\omega \left( 
\mathrm{cls}\left( tY\right) ,\mathcal{F}\right) $ attracts $\mathrm{cls}%
\left( tY\right) $, hence $\mathcal{A}$ attracts $\mathrm{cls}\left(
tY\right) $. Then, for a given $\mathcal{U}\in \mathcal{O}$, there is $A\in 
\mathcal{F}$ such that $A\mathrm{cls}\left( tY\right) \subset \mathrm{St}%
\left[ \mathcal{A},\mathcal{U}\right] $. By Hypothesis $\mathrm{H}_{3}$,
there is $B\in \mathcal{F}$ such that $B\subset At$. Then we have $BY\subset
AtY\subset \mathrm{St}\left[ \mathcal{A},\mathcal{U}\right] $. Hence $%
\mathcal{A}$ attracts $Y$, and therefore $\mathcal{A}$ is the global $%
\mathcal{F}$-attractor.
\end{proof}

We have the following consequence from Proposition \ref{1.1} and Theorem \ref%
{GlobalUniform}.

\begin{corollary}
Assume that the family $\mathcal{F}$ satisfies both Hypotheses $\mathrm{H}%
_{3}$ and $\mathrm{H}_{4}$ and the semigroup action $\left( S,X\right) $ is $%
\mathcal{F}$-eventually bounded and eventually compact. Let $\mathcal{A}%
\subset X$ be a nonempty set that is closed, compact, and invariant. Then $%
\mathcal{A}$ is the global $\mathcal{F}$-attractor for $\left( S,X\right) $
if and only if $\mathcal{A}$ is the global uniform $\mathcal{F}$-attractor
for $\left( S,X\right) $.
\end{corollary}

\section{Examples}

In this section we present illustrating examples for the setting of this
paper by exploring semigroup actions on function spaces.

Let $E$ be a normed vector space endowed with the admissible family $%
\mathcal{O}_{\mathrm{d}}$ as stated in Example \ref{Ex1}. For finite
sequences $\alpha =\left\{ x_{1},...,x_{k}\right\} $ in $E$ and $\epsilon
=\left\{ \mathcal{U}_{\varepsilon _{1}},...,\mathcal{U}_{\varepsilon
_{k}}\right\} $ in $\mathcal{O}_{\mathrm{d}}$, let $\mathcal{U}_{\alpha
}^{\epsilon }$ be the cover of $E^{E}$ given by the sets of the form $%
\tprod_{x\in \mathbb{R}^{n}}U_{x}$ where $U_{x_{i}}=\mathrm{B}\left(
a_{i},\varepsilon _{i}\right) \in \mathcal{U}_{\varepsilon _{i}}$, for $%
i=1,...,k$, and $U_{x}=E$ otherwise. The family $\mathcal{O}_{\mathrm{p}%
}=\left\{ \mathcal{U}_{\alpha }^{\epsilon }\right\} $ is a base for the
uniformity of pointwise convergence on $E^{E}$ (see e.g. \cite[Corollary
37.13]{Will}).

Let $E^{E}$ endowed with the pointwise convergence topology. Then the
inclusion map $i:E\hookrightarrow E^{E}$, where $i\left( x\right) $ is the
constant function $i\left( x\right) \equiv x$, is a continuous map. It is
well-known that $E^{E}$ is not metrizable with the pointwise convergence
topology.

We start with an example of global uniform attractor that is not global
attractor, showing that the existence of global uniform attractor does not
assure asymptotic compactness.

\begin{example}
Assume that $E$ is the $n$-dimensional Euclidean vector space and take $%
a=\left( a_{1},...,a_{n}\right) \in \left( 0,1\right) ^{n}$. Let $\mu
:E\times E^{E}\rightarrow E^{E}$ be the action of $E$ on $E^{E}$ given by 
\begin{equation*}
\mu \left( \left( t_{1},,,t_{n}\right) ,f\right) =\left(
a_{1}^{t_{1}}f_{1},...,a_{n}^{t_{n}}f_{n}\right)
\end{equation*}%
where $f=\left( f_{1},...,f_{n}\right) $. For each $k\in \mathbb{N}$, define
the set $A_{k}=\left\{ \left( t_{1},...,t_{n}\right) \in E:t_{i}\geq
k\right\} $. Consider the filter basis $\mathcal{F}=\left\{ A_{k}:k\in 
\mathbb{N}\right\} $ for limit behavior. If $t_{\lambda }\rightarrow _{%
\mathcal{F}}\infty $, it is easily seen that $t_{\lambda }f\left( x\right)
=\mu \left( t_{\lambda },f\right) \left( x\right) \rightarrow 0$, for every $%
f\in E^{E}$ and $x\in E$, that is, $i\left( 0\right) \in \omega \left( f,%
\mathcal{F}\right) $, for every $f\in E^{E}$. This means that $i\left(
0\right) \in J\left( f,\mathcal{F}\right) $ for every $f\in E^{E}$. Now,
suppose that $g\in J\left( f,\mathcal{F}\right) $. Then there are nets $%
t_{\lambda }\rightarrow _{\mathcal{F}}\infty $ and $f_{\lambda }\rightarrow
f $ such that $t_{\lambda }f_{\lambda }\rightarrow g$. Writing $t_{\lambda
}=\left( t_{1}^{\lambda },...,t_{n}^{\lambda }\right) $, $f_{\lambda
}=\left( f_{1}^{\lambda },...,f_{n}^{\lambda }\right) $, and $f=\left(
f_{1},...,f_{n}\right) $, we have $a_{i}^{t_{i}^{\lambda }}\rightarrow 0$
and $f_{i}^{\lambda }\left( x\right) \rightarrow f_{i}\left( x\right) $, for
all $x\in E$. Hence $a_{i}^{t_{i}^{\lambda }}f_{i}^{\lambda }\left( x\right)
\rightarrow 0$, and therefore $t_{\lambda }f_{\lambda }\left( x\right)
\rightarrow 0$, for all $x\in E$. This means that $g=i\left( 0\right) $, and
then $J\left( f,\mathcal{F}\right) =\left\{ i\left( 0\right) \right\} $, for
every $f\in E^{E}$. Thus $\mathcal{A}=\left\{ i\left( 0\right) \right\} $ is
the global uniform $\mathcal{F}$-attractor for $\left( E,E^{E},\mu \right) $%
. We now explain that $\left( E,E^{E},\mu \right) $ does not have global $%
\mathcal{F}$-attractor, that is, $\mathcal{A}$ is not the global $\mathcal{F}
$-attractor for $\left( E,E^{E},\mu \right) $. Take $\alpha =\left\{ \left(
0,...,0\right) \right\} $ and $\beta =\left\{ \left( 1,...,1\right) \right\} 
$ in $E$, $\epsilon =\left\{ \mathcal{U}_{1}\right\} $ in $\mathcal{O}_{%
\mathrm{d}}$, and define $Y=\mathrm{St}\left[ i\left( 0\right) ,\mathcal{U}%
_{\alpha }^{\epsilon }\right] $. Then $Y$ is bounded set, by definition. To
show that $\mathcal{A}$ is not the global $\mathcal{F}$-attractor for $%
\left( E,E^{E},\mu \right) $, it is enough to show that there is no $k\in 
\mathbb{N}$ such that $A_{k}Y\subset \mathrm{St}\left[ i\left( 0\right) ,%
\mathcal{U}_{\beta }^{\epsilon }\right] $. In fact, note that $f\in \mathrm{%
St}\left[ i\left( 0\right) ,\mathcal{U}_{\beta }^{\epsilon }\right] $
implies that $f\left( 1,...,1\right) ,0\in \mathrm{B}\left( b,1\right) $ for
some $b\in B$, hence $\left\Vert f\left( 1,...,1\right) \right\Vert <2$. For
each $k\in \mathbb{N}$, define $f_{k}\in E^{E}$ by $f_{k}\left( x\right)
=2\left( a_{1}^{-k}x_{1},...,a_{n}^{-k}x_{n}\right) $, where $x=\left(
x_{1},...,x_{n}\right) $. Then $f_{k}\left( 0\right) =0$, hence $f_{k}\in Y$%
. However, $\left\Vert \left( k,...,k\right) f_{k}\left( 1,...,1\right)
\right\Vert =\left\Vert \left( 2,...,2\right) \right\Vert >2$, hence $\left(
k,...,k\right) f_{k}\notin \mathrm{St}\left[ i\left( 0\right) ,\mathcal{U}%
_{\beta }^{\epsilon }\right] $. It follows that $\left( k,...,k\right)
f_{k}\in A_{k}Y\setminus \mathrm{St}\left[ i\left( 0\right) ,\mathcal{U}%
_{\beta }^{\epsilon }\right] $, for all $k\in \mathbb{N}$, and therefore $%
\mathcal{A}$ is not the global $\mathcal{F}$-attractor for $\left(
E,E^{E},\mu \right) $.
\end{example}

\begin{example}
\label{Ex4}Let $\mathbb{N}$ be the semigroup of positive integers with
multiplication and $\mu :\mathbb{N}\times E^{E}\rightarrow E^{E}$ the action
given by $\mu \left( n,f\right) =f^{n}$. By considering the filter basis $%
\mathcal{F}=\left\{ \left\{ k\in \mathbb{N}:k\geq n\right\} :n\in \mathbb{N}%
\right\} $, $n_{\lambda }\rightarrow _{\mathcal{F}}\infty $ in $\mathbb{N}$
means $n_{\lambda }\rightarrow +\infty $. Let $K\subset E$ be a compact set
and $X\subset E^{E}$ be the subspace of all contraction maps of $E$ with
fixed point in $K$ and same Lipschitz constant $L<1$. Then $X$ is invariant
and $i\left( K\right) \subset X$ is a compact, closed, and invariant set in $%
X$. Consider the restriction action $\mu :\mathbb{N}\times X\rightarrow X$.
We claim that $i\left( K\right) $ is the global $\mathcal{F}$-attractor for
this action. Indeed, for each $f\in X$, denote by $x_{f}$ the fixed point of 
$f$. For sequences $\alpha =\left\{ z_{1},...,z_{m}\right\} $ in $E$ and $%
\epsilon =\left\{ \mathcal{U}_{\varepsilon _{1}},...,\mathcal{U}%
_{\varepsilon _{m}}\right\} $ in $\mathcal{O}_{\mathrm{d}}$, define $\delta
=\sup \left\{ \left\Vert z_{j}-x_{f}\right\Vert :f\in X,j=1,...,m\right\} $.
Pick $n_{0}\in \mathbb{N}$ such that $L^{n_{0}}\delta <\min \left\{
\varepsilon _{1},...,\varepsilon _{m}\right\} $. Then, for any $f\in X$ and $%
n\geq n_{0}$, we have 
\begin{equation*}
\left\Vert f^{n}\left( z_{j}\right) -x_{f}\right\Vert <L^{n}\left\Vert
z_{j}-x_{f}\right\Vert <L^{n_{0}}\delta <\min \left\{ \varepsilon
_{1},...,\varepsilon _{m}\right\}
\end{equation*}%
for all $j=1,...,m$. Hence $f^{n},i\left( x_{f}\right) \in \tprod_{x\in 
\mathbb{R}^{n}}U_{x}$, where $U_{z_{j}}=\mathrm{B}\left( x_{f},\varepsilon
_{j}\right) $, for $j=1,...,m$, and $U_{x}=E$ otherwise, that is, $f^{n}\in 
\mathrm{St}\left[ i\left( x_{f}\right) ,\mathcal{U}_{\alpha }^{\epsilon }%
\right] $ for every $f\in X$ and $n\geq n_{0}$. It follows that $n\geq n_{0}$
implies 
\begin{equation*}
\rho _{i\left( K\right) }\left( nX\right) =\bigcap\limits_{f\in X}\rho
\left( f^{n},i\left( K\right) \right) \prec \bigcap\limits_{f\in X}\rho
\left( f^{n},i\left( x_{f}\right) \right) \prec \left\{ \mathcal{U}_{\alpha
}^{\epsilon }\right\} .
\end{equation*}%
Hence $\rho _{i\left( K\right) }\left( nX\right) \rightarrow \mathcal{O}_{%
\mathrm{p}}$, and therefore $i\left( K\right) $ is the global $\mathcal{F}$%
-attractor for $\left( \mathbb{N},X,\mu \right) $. Moreover, by Theorem \ref%
{GlobalUniform}, $i\left( K\right) $ is the global uniform $\mathcal{F}$%
-attractor for $\left( \mathbb{N},X,\mu \right) $ and $J\left( X,\mathcal{F}%
\right) =i\left( K\right) $.
\end{example}

\begin{example}
As a special case of Example \ref{Ex4}, let $X\subset E^{E}$ be the subspace
of all contraction operators with same Lipschitz constant $L<1$. Then $%
i\left( 0\right) $ is the global $\mathcal{F}$-attractor for $\left( \mathbb{%
N},X,\mu \right) $.
\end{example}

\begin{example}
\label{Ex3} For $K>0$ and $0<L<1$, let $X\subset E^{E}$ be the subspace of
all Lipschitzian map of $E$ with same Lipschitz constant $K$ and $S\subset
E^{E}$ the topological semigroup of contraction operators $S=\left\{ T\in 
\mathcal{L}\left( E\right) :\left\Vert T\right\Vert \leq L<1\right\} $.
Consider the action $\mu :S\times X\rightarrow X$ of $S$ on $E^{E}$ given by
the composition $\mu \left( T,f\right) =T\circ f$. For a given $m\in \mathbb{%
N}$, set $S_{m}=\left\{ T^{k}:T\in S,k\geq m\right\} $. For $p,q\in \mathbb{N%
}$, we have $S_{m}=S_{p}\cap S_{q}$, where $m=\max \left\{ p,q\right\} $.
Hence the family $\mathcal{F}=\left\{ S_{m}:m\in \mathbb{N}\right\} $ is a
filter basis on the subsets of $S$. We claim that $\mathcal{A}=\left\{
i\left( 0\right) \right\} $ is the global $\mathcal{F}$-attractor for $%
\left( S,X,\mu \right) $. In fact, for all $T\in S$ and $x\in \mathbb{R}^{n}$%
, we have $T\circ i\left( 0\right) \left( x\right) =T\left( 0\right) =0$,
hence $T\circ i\left( 0\right) =i\left( 0\right) $, for every $T\in S$, and
therefore $\mathcal{A}$ is invariant. Now, let $Y\subset E^{E}$ be a bounded
set. Then there are sequences $\beta =\left\{ x_{1},...,x_{n}\right\} $ in $%
E $ and $\delta =\left\{ \mathcal{U}_{\delta _{1}},...,\mathcal{U}_{\delta
_{n}}\right\} $ in $\mathcal{O}_{\mathrm{d}}$ such that $Y\subset \mathrm{St}%
\left[ h,\mathcal{U}_{\beta }^{\delta }\right] $ for some $h\in E^{E}$. It
follows that $\left\Vert f\left( x_{i}\right) -g\left( x_{i}\right)
\right\Vert <4\delta _{i}$ for all $f,g\in Y$ and $i=1,...,n$. Take $%
T_{\lambda }\rightarrow _{\mathcal{F}}\infty $ and sequences $\alpha
=\left\{ z_{1},...,z_{m}\right\} $ in $E$ and $\epsilon =\left\{ \mathcal{U}%
_{\varepsilon _{1}},...,\mathcal{U}_{\varepsilon _{m}}\right\} $ in $%
\mathcal{O}_{\mathrm{d}}$. For any $f,g\in Y$ and $j=1,...,m$, we have 
\begin{eqnarray*}
\left\Vert f\left( z_{j}\right) -g\left( z_{j}\right) \right\Vert &\leq
&\left\Vert f\left( z_{j}\right) -f\left( x_{1}\right) \right\Vert
+\left\Vert f\left( x_{1}\right) -g\left( x_{1}\right) \right\Vert
+\left\Vert g\left( x_{1}\right) -g\left( z_{j}\right) \right\Vert \\
&<&2K\left\Vert z_{j}-x_{1}\right\Vert +4\delta _{1}.
\end{eqnarray*}%
Hence the set $\left\{ f\left( z_{j}\right) :f\in Y,j=1,...,m\right\} $ is
bounded in $E$. Suppose that $\left\Vert f\left( z_{j}\right) \right\Vert
\leq M$ for all $f,g\in Y$ and $j=1,...,m$. Since $\left\Vert T_{\lambda
}\right\Vert \leq L<1$, for a given $k\in \mathbb{N}$ there is $\lambda _{0}$
such that $\lambda \geq \lambda _{0}$ implies 
\begin{equation*}
\left\Vert T_{\lambda }\left( x\right) -T_{\lambda }\left( y\right)
\right\Vert \leq L^{k}\left\Vert x-y\right\Vert
\end{equation*}%
for all $x,y\in E$. Then, for any $f\in Y$, we have 
\begin{equation*}
\left\Vert T_{\lambda }\left( f\left( z_{j}\right) \right) \right\Vert \leq
L^{k}\left\Vert f\left( z_{j}\right) \right\Vert \leq L^{k}M
\end{equation*}%
for all $j=1,...,m$ and $\lambda \geq \lambda _{0}$. Take $k\in \mathbb{N}$
such that $L^{k}M<\min \left\{ \varepsilon _{1},...,\varepsilon _{m}\right\} 
$. Then there is $\lambda _{0}$ such that $\lambda \geq \lambda _{0}$
implies $\left\Vert T_{\lambda }f\left( z_{j}\right) \right\Vert <\min
\left\{ \varepsilon _{1},...,\varepsilon _{m}\right\} $ for all $j=1,...,m$
and $f\in Y$, hence $T_{\lambda }f,i\left( 0\right) \in \tprod_{x\in \mathbb{%
R}^{n}}U_{x}$, where $U_{z_{j}}=\mathrm{B}\left( 0,\varepsilon _{j}\right) $%
, for $j=1,...,m$, and $U_{x}=E$ otherwise, that is, $T_{\lambda }f\in 
\mathrm{St}\left[ i\left( 0\right) ,\mathcal{U}_{\alpha }^{\epsilon }\right] 
$. This means that $\mathcal{U}_{\alpha }^{\epsilon }\in \rho \left(
T_{\lambda }Y,i\left( 0\right) \right) $ whenever $\lambda \geq \lambda _{0}$%
, and then $\rho _{\mathcal{A}}\left( T_{\lambda }Y\right) \rightarrow 
\mathcal{O}_{\mathrm{p}}$. Therefore $\mathcal{A}$ is the global $\mathcal{F}
$-attractor for $\left( S,E^{E},\mu \right) $. Moreover, $\mathcal{A}$ is
the global uniform $\mathcal{F}$-attractor for $\left( S,X,\mu \right) $ and 
$J\left( X,\mathcal{F}\right) =\left\{ i\left( 0\right) \right\} $.

This example can be easily extended by taking the topological semigroup $%
S\subset E^{E}$ of all contractions with same fixed point $x_{0}$ and same
Lipschitz constant $L<1$. In this case, $\mathcal{A}=\left\{ i\left(
x_{0}\right) \right\} $ is the global $\mathcal{F}$-attractor for $\left(
S,X\right) $.
\end{example}

\end{document}